\documentclass[reqno]{amsart}

\usepackage{graphicx,color}
\theoremstyle{plain}

\newtheorem{theorem}{Theorem}[section]
\newtheorem{lemma}[theorem]{Lemma}
\newtheorem{remark}[theorem]{Remark}
\newtheorem{proposition}[theorem]{Proposition}

\numberwithin{equation}{section}
\allowdisplaybreaks[1]

\theoremstyle{definition}

\theoremstyle{remark}

\newcommand{\BC}{{\mathbb C}}

\newcommand{\cB}{{\mathcal B}}
\newcommand{\cC}{{\mathcal C}}
\newcommand{\cF}{{\mathcal F}}
\newcommand{\cH}{{\mathcal H}}

\newcommand{\cK}{{\mathcal K}}\newcommand{\cL}{{\mathcal L}}

\newcommand{\cU}{{\mathcal U}}
\newcommand{\cX}{{\mathcal X}}
\newcommand{\cY}{{\mathcal Y}}\newcommand{\cZ}{{\mathcal Z}}

\newcommand{\bE}{{\mathbf E}}
\newcommand{\bG}{{\mathbf G}}

\newcommand{\bP}{{\mathbf P}}
\newcommand{\bR}{{\mathbf R}}
\newcommand{\bS}{{\mathbf S}}
\newcommand{\bU}{{\mathbf U}}\newcommand{\bV}{{\mathbf V}}

\newcommand{\al}{\alpha}
\newcommand{\be}{\beta}

\newcommand{\Ga}{\Gamma}
\newcommand{\de}{\delta}
\newcommand{\De}{\Delta}

\newcommand{\la}{\lambda}

\newcommand{\si}{\sigma}

\newcommand{\bDe}{{\boldsymbol\Delta}}

\newcommand{\ov}[1]{{\overline{#1}}}

\newcommand{\inn}[2]{\ensuremath{\langle #1,#2 \rangle}}

\newcommand{\wtil}{\widetilde}

\newcommand{\BDe}{{\boldsymbol \Delta}}

\newcommand{\bff}{{\bf f}}
\newcommand{\br}{{\bf r}}
\newcommand{\bs}{{\bf s}}

\newcommand{\bDelta}{{\boldsymbol \Delta}}

\begin{document}

\title[The free noncommutative setting]{Bounded Real Lemma and
structured singular value versus diagonal scaling: the free
noncommutative setting}

\author[J.A.~Ball]{Joseph A. Ball}
\address{Department of Mathematics,
Virginia Tech,
Blacksburg, VA 24061-0123, USA}
\email{joball@math.vt.edu}

\author[G.~Groenewald]{Gilbert J.~Groenewald}
\address{Department of Mathematics, Unit for BMI, North-West University,
Potchefstroom, 2531 South Africa}
\email{Gilbert.Groenewald@nwu.ac.za}

\author[S.~ter Horst]{Sanne Ter Horst}
\address{Department of Mathematics, Unit for BMI, North-West University,
Potchefstroom, 2531 South Africa}
\email{Sanne.TerHorst@nwu.ac.za}



\begin{abstract}
The structured singular value (often referred to simply as $\mu$) was
introduced independently by Doyle and Safanov as a tool for analyzing
robustness of system stability and performance in the presence of
structured uncertainty in the system parameters.  While the
structured singular value provides a necessary and sufficient
criterion for robustness with respect to a structured ball of
uncertainty, it is notoriously difficult to actually compute. The
method of diagonal (or simply "D") scaling, on the other hand,
provides an easily computable upper bound (which we call
$\widehat \mu$) for the structured singular
value, but provides an exact evaluation of $\mu$ (or even a useful
upper bound for $\mu$) only in special cases.  However it was
discovered in the 1990s that a certain enhancement of the uncertainty structure
(i.e., letting the uncertainty parameters be freely
noncommuting linear operators on an infinite-dimensional separable Hilbert
space) resulted in the $D$-scaling procedure leading to an exact
evaluation of $\mu_{\text{enhanced}}$ ($\mu_{\text{enhanced}} =
\widehat \mu$), at least for the tractable special cases which were
analyzed in complete detail.
On the one hand this enhanced uncertainty has
some appeal from the physical point of view: one can allow the
uncertainty in the plant parameters to be time-varying, or more
generally, one can catch the uncertainty caused by the designer's
decision not to model the more complex (e.g.~ nonlinear) dynamics of
the true plant.  On the other hand, the precise mathematical
formulation of this enhanced uncertainty structure makes contact with
developments in the growing theory of analytic functions in freely
noncommuting arguments and associated formal power series in freely
noncommuting indeterminates. In this article we obtain the
$\widetilde \mu = \widehat \mu$ theorem for a more satisfactory
general setting.
\end{abstract}

\maketitle

\section{Introduction}\label{S:intro}
The structured singular value was introduced
independently by Doyle \cite{DoyleMu} and Safanov \cite{SafanovMu};
see \cite{ZDG} for a thorough more recent treatment.  Let $N$ be a
positive integer with a partitioning $N = n_1 + \cdots + n_\bs + m_1 +
\cdots + m_\bff$ for positive integers $n_i$ ($i = 1, \dots, \bs$) and
$m_j$ ($j=1, \dots, \bff$).  We let $\bDelta$ denote the set of $N
\times N$ matrices of the form
\begin{align}
\bDelta =   \{ {\rm   diag} [ \delta_1 I_{n_1}, \dots, \delta_\bs
I_{n_\bs}, \Delta_1, \dots, \Delta_\bff ] \colon  \delta_i \in {\mathbb C},\,
 \Delta_j \in {\mathbb C}^{m_j \times m_j} \}.
 \label{structure}
\end{align}
For an $N \times N$ matrix $M \in {\mathbb C}^{N \times N}$, we
define the {\em structured singular value} of $M$ with respect to $\bDe$ by
\begin{equation}  \label{ssv}
\mu_{\bDelta}(M) : = \frac{1}{ {\rm min} \{ \| \Delta \| \colon
\Delta \in \bDelta, \, 1 \in \sigma(M \Delta) \} },
\end{equation}
where in general $\sigma(X)$ denotes the spectrum of the square
matrix $X$.  Motivation for this notion comes from robust control
theory (see \cite{ZDG, DP}).

In the case where $\bs=0$ and $\bff = 1$,
the structured singular value $\mu_{\bDelta}(M)$ collapses to the
largest singular value $\overline{\sigma}_1(M)$ of $M$ or,
equivalently, the induced operator norm of $M$ as an operator on
${\mathbb C}^N$, where ${\mathbb C}^N$ is given the standard
$2$-norm.  A key property of the largest singular value from the
point of view of systems and control follows from the Small Gain
Theorem.\smallskip

\begin{theorem}[Small Gain Theorem]\label{T:SmallGain}
{\sl Let $M \in {\mathbb C}^{N \times N}$ such that
$\overline{\si}_1(M)<1$. Then $I - \De M$ is invertible for all $\De\in\BC^{N\times N}$
with $\|\De\| \le 1$.}\smallskip
\end{theorem}

The systems and control interpretation of this result is that
$\overline{\si}_1(M)<1$ implies that perturbation of the `plant' $M$
with a multiplicative perturbation $\Delta$ does not affect stability
of the closed-loop feedback as long as $\| \Delta \| \le  1$.

Another well known case is when $\bs=1$ and $\bff=0$. In that case
$\mu_{\bDelta}(M)$ coincides
with the spectral radius of $M$, and hence $\mu_{\bDelta}(M)<1$ implies
that $I-\de M$ is
invertible for all $\de\in\BC$ with $|\de|\leq 1$.

There are many applications in which the uncertainty parameter
$\Delta$ is known to carry some structure, as in \eqref{structure}.
In these cases it is enough that the structured singular value
$\mu_{\bDelta}(M)$ be less than 1 to guarantee the maintenance of
stability against structured multiplicative perturbations $\Delta \in
\bDelta$ with $\| \Delta \| \le 1$.

However, it turns out that the structured singular value
$\mu_{\bDelta}(M)$ is notoriously difficult to compute in a
computationally efficient and reliable way. Indeed, computing the
exact structured singular value $\mu_{\bDelta}(M)$ is an NP-hard
problem \cite{BYDM}.

There is a convenient upper bound for $\mu_{\bDelta}(M)$ defined by
$$
\widehat \mu_{\bDelta}(M) := \inf \{ \| D M D^{-1}\| \colon D \in
\bDelta' \text{ and } D \text{ invertible}\},
$$
where $\bDelta'$ denotes the commutant of $\bDelta$ in $\BC^{N\times N}$, that is,
\begin{equation}\label{commutant}
\bDelta' = \{ D \in {\mathbb C}^{N \times N} \colon D \De = \De
D \text{ for all } \De \in \bDelta\}.
\end{equation}
It turns out that $\widehat \mu_{\bDelta}(M)$ can be computed
accurately and efficiently. Indeed, to test whether $\widehat
\mu_\bDe(M) < 1$ it suffices to find a positive definite matrix $X \in
\bDelta'$ which solves the structured Stein inequality
$$
  M^* X M - X \prec  0.
$$
Note that the condition $X \in \bDelta'$ is equivalent to $X$ having
the block diagonal form
$$
  X = {\rm diag}[ X_1, \dots, X_\bs, x_1 I_{m_1}, \dots, x_\bff
  I_{m_\bff} ],
$$
where $X_i$ is a positive definite matrix of size $n_i\times n_i$
(for $i=1,\ldots, \bs$) and $x_j$ a positive number (for
$j=1,\ldots,\bff$). This puts the computation of $\widehat \mu_{\bDe}$
within the framework of the {\tt MATLAB} LMI toolbox.

While the general inequality $\mu_{\bDelta}(M) \le \widehat
\mu_{\bDelta}(M)$ is easily derived,  actual equality holds only in
very special cases. In particular, equality holds for all $M$ with
respect to a given choice of structure specified by nonnegative
integers $\bs$ and $\bff$ as in \eqref{structure} if and only if
$2\bs + \bff \le 3$ (see \cite{PD, ZDG, DP}). Moreover, even with
$\bs$ and $\bff$ in \eqref{structure} fixed, there is in general no bound on  the gap
between $\mu_{\bDelta}(M)$ and its upper bound $\widehat
\mu_{\bDelta}(M)$; see \cite{Treil}. Thus the compromise of using
$\widehat \mu_{\bDelta}(M)$ as a substitute for $\mu_{\bDelta}(M)$
can be arbitrarily conservative.

However, if the structure is relaxed by letting the uncertainty parameters $\delta_i$
and the matrix entries of $\Delta_j$ be operators on a separable infinite-dimensional
Hilbert space,
say on $\ell^2=\ell^{2}({\mathbb Z}_{+})$, the Hilbert space of square-summable complex
sequences
indexed by the nonnegative integers ${\mathbb Z}_{+}$. Then the modified $\mu$ is equal
to its
easily computable upper bound. To make this precise, we introduce the enhanced structure
\begin{equation}\label{enhanced}
\widetilde{\bDelta} = \{ {\rm diag} [
\widetilde{\delta}_1 \otimes I_{{\mathbb C}^{n_1}},\dots,
\widetilde{\delta}_\bs\otimes I_{{\mathbb C}^{n_\bs}},
\widetilde \Delta_1, \dots, \widetilde \Delta_\bff ] \}
\end{equation}
where each $\widetilde \delta_i \in \cL(\ell^2)$ and each $\widetilde \Delta_j
\in \cL(\ell^2_{m_j})$,
with $\ell^2_{m_j}=\BC^{m_j}\otimes\ell^2$. We replace $M \in
{\mathbb C}^{N \times N}$ with $\widetilde M = M\otimes I_{\ell^2}
\in \cL(\ell_N^2)$ and define a new variation on $\mu(M)$ by
$$
 \widetilde \mu_{\bDelta}(M)  : =
\mu_{\widetilde{\bDelta}}(I_{\ell^{2}} \otimes M).
$$
It turns out that the two notions of $\widehat \mu$ are the same:
$$
   \widehat \mu_{\widetilde{\bDelta}}(I_{\ell^{2}} \otimes M ) =
\widehat \mu_{\bDelta}(M).
$$
and hence the common value $\widehat \mu_{\bDelta}(M)$ is easily
computable.  The remarkable result is that this relaxed structured
singular value is always equal to its easily computable upper bound,
i.e.,
\begin{equation}   \label{equality}
  \widetilde \mu_{\bDelta}(M) =  \widehat \mu_{\bDelta}(M).
\end{equation}
This result can be found in the dissertation of Paganini
\cite{Paganini} and is summarized in \cite{LZD} without proof;
the complete proof, as thoroughly elucidated in the
book \cite{DP} (at least for the case where $\bs = 0$ with the case
$\bs > 0$ indicated in the exercises) draws on earlier
ideas and results from Megretski-Treil \cite{MT} and Shamma
\cite{Shamma}. Also there is an interpretation of the quantity
$\widetilde \mu$ as robustness with respect to an enlarged
block-structured uncertainty; one can view
this enhanced block-structured uncertainty as allowing time-varying
uncertainty in the system parameters, or, perhaps more appealingly,
as specifying a range for the input-output pairs of the true plant,
thus allowing for unmodeled dynamics (e.g.~nonlinearities) in the
behavior of the true plant (see \cite[Chapter 8]{DP} for more
complete details).

We mention that this result is but one more instance of a general
phenomenon appearing often of late where a single-variable function
theory result fails to have a compelling or complete generalization to
the commutative multivariable setting, but does have a clean
complete generalization to the free noncommutative setting;  as for
other examples, we mention the realization theory for rational
matrix functions and for the Schur class on the unit disk (see
\cite{BGM1, BGM2, AMcCGlobal}), Helton's result on representing a polynomial as a
sum of squares \cite{Helton},  recent results in free
noncommutative real algebraic geometry \cite{CHKMcCN, HMcC2004},
results on proper analytic maps \cite{HKMcCS2009, HKMcC2011}, as well
as convexity theory \cite{HMcC2012, HMcCV2006} and
Nevanlinna-Pick interpolation \cite{AMcCPick}.

As elegant as this result is, it is incomplete from a conceptual
point of view since the structure given by \eqref{structure} is
limited in two respects:
\begin{itemize}
\item[(L1)]
There is an asymmetry between the scalar blocks and the full blocks
in \eqref{structure}.   A scalar block $\delta_i I_{n_i}$ can be
considered as a full block with size $m_i = 1$, but with a repetition
(or multiplicity) of $n_i$ possibly larger than $1$ allowed. On the
other hand, the full blocks $\Delta_j$ are considered to be
independently arbitrary with no repetitions allowed.

\item[(L2)]
All blocks are considered square.  There are interesting
multidimensional input/state/output systems where this same structure
occurs but with
nonsquare blocks (see \cite{BGM1, BGM2}).
\end{itemize}

These limitations were addressed in the work of Ball-Groenewald-Malakorn
\cite{BGM3} by making a connection with the earlier work of the same
authors on the realization theory for so-called Structured
Noncommutative Multidimensional Linear Systems (SNMLSs), including a
Kalman decomposition and state-space similarity theorem \cite{BGM1},
together with a realization theorem for a noncommutative Schur-Agler
class associated with conservative SNMLSs \cite{BGM2}.  The structure
of a SNMLS was encoded in an admissible graph $\bG$, i.e.,
bipartite graph $G$ carrying some additional structure together with a
multiplicity function; see Section \ref{S:GraphSetup} for the precise setup.
Motivation for introduction of this framework came from the quest for a more
convenient coordinate-free way to analyze structures $\bDelta$ as in
\eqref{structure} with the limitations (L1) and (L2) removed.  The
idea in \cite{BGM3} was to identify the resolvent expression $\Delta
\mapsto (I - \Delta M)^{-1}$ as an element of the associated
Schur-Agler class $\mathcal{SA}_{\bG}(\cU)$ in case $\mu_{\bG}(M) < 1$.
However, this identification required an unnecessary additional hypothesis making
the analysis in \cite{BGM3} incomplete.  One of the contributions of the present
paper is to adapt  one piece of the analysis in \cite{Paganini, DP} to
verify a key lemma (see Lemma \ref{L:DP}) which implies that this additional hypothesis
indeed can be removed  and thereby to
complete the analysis begun in \cite{BGM3}.

A second contribution of the present paper is to identify the extra ingredient needed to
show how the techniques of Dullerud-Paganini \cite{Paganini, DP} can
be adapted to get \eqref{equality} in full generality (without the
limitations (L1) and (L2)); the precise result is formulated in our Main Result (Theorem \ref{T:I}).

We also show how our Main Result itself
can be used to get an alternative proof of the realization theorem for the
noncommutative Schur-Agler class $\mathcal{SA}_{\bG}(\cU, \cY)$, at
least for the finite-dimensional case (see Remark \ref{R:mu-to-ncBRL});
thus one can argue that the
main result of \cite{BGM2} was already implicitly contained in the 1996
dissertation of Paganini \cite{Paganini}.  It is interesting to note
that the proof based on \cite{BGM3}  requires the
realization theorem for the noncommutative Schur-Agler class
$\mathcal{SA}_{\bG}(\cU, \cY)$ which ultimately relies on an
infinite-dimensional cone-separation argument, while the proof of
Dullerud-Paganini \cite{Paganini, DP} uses a more
elementary finite-dimensional cone-separation argument.
We should also mention that relatively recent results of
K\"oro\u{g}lu-Scherer \cite{KS2007} also remove the limitations (L1)
and (L2) and present still finer results concerning robust
stability/performance against a fine class of structured
uncertainties $\bDelta$ (see Remark \ref{R:slowlytimevarying} below).

The paper is organized as follows.
Section \ref{S:prelim} reviews notation and results concerning tensor
product spaces which will be needed in the sequel; this includes an
adaptation of the Douglas lemma to the higher multiplicity setup,
which is the extra ingredient needed to carry out the
Dullerud-Paganini proof of the Main Result for the higher
multiplicity situation.
In Section \ref{S:GraphSetup} we
recall the graph formalism from \cite{BGM1, BGM2, BGM3} and
reformulate the desired result \eqref{equality} in this framework for
the general setting.  In Section \ref{S:BRLproof} we identify and
prove the key lemma needed to complete the analysis
from \cite{BGM3} and thereby get our first proof of the Main Result, Theorem \ref{T:I} below.
In Section \ref{S:DPproof} we show how the
analysis of Dullerud-Paganini can be beefed up to handle the more
general case with limitations (L1) and (L2) removed.  In Section
\ref{S:BFKT} we show how the alternative enhanced structured singular value of
Bercovici-Foias-Khargonekar-Tannenbaum \cite{BFKT} can be handled by
the same type of convexity analysis as used by Dullerud-Paganini.

A preliminary version of this report was given in the conference
proceedings paper \cite{BGtH}.

\section{Preliminaries on tensor products}  \label{S:prelim}

Let $\cH$ and $\cK$ be two Hilbert spaces. We shall have use for a
fixed conjugation operator $\cC$ on $\cK$, i.e., an operator $\cC$ on
$\cK$ with the following properties:
\begin{itemize}
\item[(i)] $\cC(\al f+g)=\bar{\al}\cC(f)+\cC(g)$ \quad (anti-linear)

\item[(ii)] $\inn{\cC f}{\cC g}=\inn{g}{f}=\overline{\inn{f}{g}}$ \quad
(isometric)

\item[(iii)] $\cC^2=I$\quad (involution)
\end{itemize}
To construct such an operator, choose any orthonormal basis $\{e_{j}
\colon j \in A\}$ for $\cK$ and define $\cC$ by
$$
  \cC \colon \sum_{j \in A} c_{j} e_{j} \mapsto \sum_{j \in A}
  \overline{c}_{j} e_{j}
$$
where $\overline{c}_{j}$ is the ordinary complex conjugate of the
complex number $c_{j}$.  For convenience of notation we shall often
write $\overline{k}$ instead of $\cC k$.

The Hilbert space tensor
product $\cH \otimes \cK$ is defined as the completion of the linear
span of the pure tensor elements $h \otimes k$ where the inner
product on pure tensors is given by
$$
\langle h \otimes k,\, h' \otimes k' \rangle_{\cH \otimes \cK} =
\langle h,h'\rangle_{\cH}\,  \langle k, k' \rangle_{\cK}.
$$
We note that in this construction the pure tensor $ch \otimes k$ is
identified with the pure tensor $h \otimes ck$ for $c \in {\mathbb
C}$ a scalar.
It is convenient to view a vector $h$ in the Hilbert space $\cH$ also
as an operator $h \in \cL({\mathbb C}, \cH)$:
$$
  h \colon c \mapsto c \cdot h \in \cH \text{ for } c \in {\mathbb C}.
$$
with adjoint $h^{*} \colon \cH \to {\mathbb C}$ given by
$$
   h^{*} \colon h' \mapsto \langle h, h' \rangle_{\cH} \in {\mathbb C}.
$$
With this interpretation, the Hilbert space inner product itself can
be rewritten as
$$
\langle h, h' \rangle_{\cH} = (h')^{*} h.
$$

A space closely related to the Hilbert space tensor product $\cH
\otimes \cK$ is the space $\cC_{2}(\cK, \cH)$ of Hilbert-Schmidt
operators from $\cH$ into $\cK$, i.e., the space of operators $T \in
\cL(\cK, \cH)$ such that $T^{*}T$ is in the trace class $\cC_1(\cK)=\cC_1(\cK,\cK)$.  These
operators form a Hilbert space with inner product given by
$$
   \langle S, T \rangle_{\cC_{2}(\cK, \cH)} = {\rm tr}(T^{*}S).
$$
In fact, the following  result gives a useful identification
between the tensor-product Hilbert space $\cH \otimes \cK$ and the
Hilbert space of Hilbert-Schmidt operators $\cC_{2}(\cK, \cH)$. For
completeness we include an elementary proof; a good reference for
more general tensor-product constructions is the book of Takesaki
\cite{Takesaki}.

\begin{proposition}   \label{P:tensor}
    Let $\cH$ and $\cK$ be two Hilbert spaces with a fixed
    conjugation operator $\cC \colon k \mapsto \overline{k}$ given on
    $\cK$.  Define a map $U_{\cH, \cK}$ on pure tensors in $\cH
    \otimes \cK$ into rank-1 operators from $\cK$ into $\cH$
    according to the formula
$$
U_{\cH, \cK} \colon h \otimes k \mapsto h (\overline{k})^{*} =: h
k^{\top}.
$$
Then $U_{\cH, \cK}$ extends by linearity and continuity to a unitary
map from the Hilbert space $\cH \otimes \cK$ onto the Hilbert space
$\cC_{2}(\cK, \cH)$.
\end{proposition}

\begin{proof}
    For purposes of the proof, we abbreviate $U_{\cH, \cK}$ to $U$.
    As $\cH \otimes \cK$ is the Hilbert space completion of the span
    of the pure tensors and $\cC_{2}(\cK, \cH)$ is the Hilbert space completion
    of the span of the rank-one operators, it suffices to check that
    $U$ preserves the respective inner products on pure tensors:
\begin{equation}  \label{check}
    \langle U [ h \otimes k], \, U [ h' \otimes k']
    \rangle_{\cC_{2}(\cK, \cH)} = \langle h \otimes k, \, h' \otimes
    k' \rangle_{\cC_{2}(\cK, \cH)}.
\end{equation}
To this end, we compute
\begin{align*}
  &  \langle U [ h \otimes k], \, U [ h' \otimes k']
    \rangle_{\cC_{2}(\cK,\cH)}  = \langle h \overline{k}^{*},
    \, h' (\overline{k'})^{*} \rangle_{\cC_{2}(\cK , \cH)}
     = {\rm tr} (\overline{k'} (h')^{*}  h \overline{k}^{*})  \\
    &  \quad   = {\rm tr}((h')^{*} h \overline{k}^{*} \overline{k'} )
     = \langle h, h' \rangle_{\cH} \cdot \langle \overline{k'},  \overline{k} \rangle_{\cK}
    = \langle h, h' \rangle_{\cH} \cdot \langle k, k' \rangle_{\cK} \\
    & \quad  =  \langle h \otimes k,\, h' \otimes k' \rangle_{\cH \otimes \cK}
\end{align*}
as required.
\end{proof}

Given  four Hilbert spaces $\cH, \cK, \cH_{0}, \cK_{0}$ and operators
$X \in \cL(\cH, \cK)$ and $Y \in \cL(\cH_{0}, \cK_{0})$, the
tensor-product operator $X \otimes Y$ is defined on pure tensors in
$\cH \otimes \cH_{0}$ according to the formula
\begin{equation}   \label{tensoropnorm}
  X \otimes Y \colon h \otimes h_{0} \mapsto Xh \otimes Yh_{0} \in \cK
  \otimes \cK_{0}.
\end{equation}
It is not hard to see that $X \otimes Y$ extends to a bounded operator
from $\cH \otimes \cH_{0}$ into $\cK \otimes \cK_{0}$ with
$\| X \otimes Y \|_{\cL(\cH \otimes \cK)}  = \| X \|_{\cL(\cH, \cK)}
\cdot \| Y \|_{\cL(\cH_{0}, \cK_{0})}$.
A convenient tool
for working with such operators is to use the identification maps
$U_{\cH, \cK}$  and $U_{\cH_{0}, \cK_{0}}$ to view $X \otimes Y$ as acting
between Hilbert-Schmidt operator spaces $\cC_{2}(\cH_{0}, \cH)$ and
$\cC_{2}(\cK_{0}, \cK)$ instead; indeed this is
one approach to seeing why $X \otimes Y$ is bounded with norm as in
\eqref{tensoropnorm}.  Here we use the notation $Y^{\top}$ for the
operator
$$
  Y^{\top} \colon k \mapsto \overline{Y^{*} \overline{k}}.
$$

\begin{proposition}  \label{P:tensor2}
Given $X \in \cL(\cH,\cK)$ and $Y \in \cL(\cH_{0}, \cK_{0})$, let $L_{X}$ be the left
multiplication
operator $L_{X} \colon T \mapsto XT$ mapping the Hilbert-Schmidt-operator space
$\cC_{2}(\cK_{0}, \cH)$ to the Hilbert-Schmidt-operator space $\cC_{2}(\cK_{0}, \cK)$,
and let $R_{Y^{\top}}$ be the right multiplication operator
$R_{Y} \colon T' \mapsto T' Y^{\top}$
mapping the Hilbert-Schmidt-operator space $\cC_{2}(\cH_{0}, \cH)$ to the
Hilbert-Schmidt-operator space $\cC_{2}(\cK_{0}, \cH)$. If $U_{\cH, \cH_{0}}
\colon \cH \otimes
    \cH_{0} \to \cC_{2}(\cH_{0}, \cH)$ and $U_{\cK, \cK_{0}} \colon
    \cK \otimes \cK_{0} \to \cC_{2}(\cK_{0}, \cK)$ are the
    identification maps as introduced
    in Proposition \ref{P:tensor}, then we have the intertwining
    relation
 \begin{equation}  \label{intertwine}
	U_{\cK, \cK_{0}} (X \otimes Y) =  L_{X} R_{Y^{\top}} U_{\cH,
	\cH_{0}}.
\end{equation}
\end{proposition}

\begin{proof}
It suffices to verify that the relation \eqref{intertwine} holds when
applied to an elementary tensor $h \otimes h_{0}$.  We compute
\begin{align*}
    & 	U_{\cK, \cK_{0}} [(X \otimes Y) (h \otimes h_{0})] =
    U_{\cK, \cK_{0}}[Xh \otimes Y h_{0}] = (Xh) (\overline{Y
    h_{0}})^{*} = (Xh)(\overline{h_{0}}^{*}
    Y^{\top})  \\
 & \quad =  X (h \overline{h_{0}}^{*}) Y^{\top}  = L_{X} R_{Y^{\top}}
 U_{\cH, \cH_{0}}[h \otimes h_{0}]
\end{align*}
as required.
 \end{proof}

 The well-known Douglas lemma (see \cite{Douglas}) asserts that,
 given Hilbert space operators $A \in \cL(\cY, \cZ)$ and $B \in
\cL(\cX, \cZ)$, there exists an operator $X \in \cL(\cX,\cY)$ with
$AX=B$ and $\| X \| \le 1$ if and only if $BB^* - AA^* \preceq 0$.
We shall have use of the adjoint version: {\em  given Hilbert space
operators $A \in \cL(\cZ, \cY)$ and $B \in \cL(\cZ, \cX)$, then there
exists an operator $X \in \cL(\cY, \cZ)$ satisfying $XA=B$ with $\| X
\| \le 1$ if and only if $B^{*}B - A^{*}A \preceq 0$.}  The special
case where $\cZ = {\mathbb C}$ appears as Lemma 8.4 in \cite{DP} and
is crucial for the proof of the multiplicity-one special case of
Theorem~\ref{T:I} there.  The following structured version of the
Douglas lemma is crucial for the second proof of our main result,
Theorem \ref{T:I}, for the general case.

\begin{proposition}  \label{P:Douglas}
    Suppose that we are given three Hilbert spaces $\cH$, $\cK$,
    $\cH_{0}$, along with vectors $p \in \cH \otimes \cH_{0}$ and $q
    \in \cK \otimes \cH_{0}$.  Then there exists an operator $X \in
    \cL(\cH, \cK)$ satisfying
 \begin{equation}   \label{problem}
 (X \otimes I_{\cH_{0}}) p = q \text{ and } \| X \|_{\cL(\cH, \cK)}
 \le 1
 \end{equation}
 if and only if
 \begin{equation}  \label{solcrit}
 U_{\cK, \cH_{0}}[q]^{*} U_{\cK, \cH_{0}}[q] - U_{\cH,
 \cH_{0}}[p]^{*} U_{\cH, \cH_{0}}[p] \preceq 0.
 \end{equation}
\end{proposition}

\begin{proof}
    Application of the identification maps $U_{\cH, \cH_{0}}$ and
    $U_{\cK, \cH_{0}}$ combined with the intertwining relation
    \eqref{intertwine} given by Proposition \ref{P:tensor2}  transforms the
    problem  of finding an $X$ satisfying \eqref{problem} to:  {\em find $X \in \cL(\cH, \cK)$ with
    $\|X\|_{\cL(\cH, \cK)} \le 1$ so that}
    $$
 X U_{\cH, \cH_{0}}[p] = U_{\cK, \cH_{0}} [q].
 $$
 The criterion for a solution of this problem is then given by the
 standard Douglas lemma (in adjoint form) resulting in
 \eqref{solcrit} as the criterion for the existence of a solution.
 \end{proof}

\section{The graph formalism} \label{S:GraphSetup}

For the remainder of this paper let $G=(\bV, \bE)$ be a finite simple undirected bipartite
graph such that
each path-connected component of $G$ is a complete bipartite graph.
Here $\bV$ denotes the set of vertices and $\bE$ the set of edges. Since
$G$ is a bipartite graph, the vertex set $\bV$ admits a decomposition
$\bV=\bS\cup \bR$, with $\bS\cap \bR=\emptyset$, such that
each edge $e\in \bE$
has one vertex in $\bS$ (the source side), denoted by $s(e)$, and one vertex in $\bR$ (the
range side),
denoted by $r(e)$. We let $\bP$ denote the set of path-connected components of $G$.
We let $\bP$
denote
the set of path-connected components of $G$.  For a vertex
$v\in \bV$ we let $[v]$
 indicate the path-connected component $p \in \bP$ that contains $v$.
For each $p \in \bP$ we denote the vertex set and edge set of
$p$ by $\bV_p$ and $\bE_p$,
respectively. Each path-connected component $p$ of $G$ is also a
simple bipartite graph and its vertex set $\bV_p$ can be decomposed as
$\bV_p=\bS_{p}\cup \bR_p$ with $\bS_{p}=\bS\cap\bV_{p}$ and $\bR_p=
\bR\cap \bV_p$. By the
assumption that each path-connected component is a complete bipartite
graph, for each $p \in \bP$, the set $\bE_p$ consists of all possible
edges connecting a vertex in $\bS_{p}$ with a vertex in $\bR_{p}$.  By
definition of connected component, no edge $e$ of $G$ connects a vertex in
$\bS_p \cup \bR_p$ with a vertex in
$\bS_{p'}  \cup \bR_{p'}$  if $p \ne p'$.

We shall on occasion want also to specify a {\em multiplicity structure} to such a
graph $G$; by this we mean a specification of a Hilbert space
$\cH_{p}$ for each path-connected
component $p \in \bP$ of $G$.  We then say that the
whole collection $\bG = (G,  \{\cH_{p} \colon p \in \bP \})$ is
an {\em admissible graph with multiplicity}, or an {\em $M$-graph} for
short.  Finally, for the most general version of the structure, we
will specify a $C^{*}$-algebra $\bDelta_{p}$ represented concretely
as a $C^{*}$-subalgebra of $\cL(\cH_{p})$;  we call this more
elaborate structure $\overline{\bG} = (G, \{\bDelta_{p} \subset
\cL(\cH_{p})\})$ a {\em admissible graph with specified
$C^{*}$-algebras}, or {\em A-graph} for short.

\subsection{The uncertainty structure: general case}  \label{S:uncertainty}
Let $\overline{\bG} = (G, \{ \bDelta_{p} \subset \cL(\cH_{p}\})$ be
an A-graph as defined above.
We set $\cH_v=\cH_{[v]}$ for each $v\in V$ and we further introduce
the spaces
\begin{equation}\label{KSKR}
\begin{aligned}
    &\cH_\bS=\bigoplus_{s \in \bS} \cH_s,\quad \cH_{\bS,p}=\bigoplus_{s\in
\bS_p}\cH_s \quad (p \in \bP),\\
&\cH_\bR=\bigoplus_{r\in \bR}\cH_r,\quad \cH_{\bR,p}=\bigoplus_{r\in
\bR_p}\cH_r\quad (p \in \bP).
\end{aligned}
\end{equation}
For $s\in \bS$ we write $\iota_{s}$ for the canonical embedding  of
$\cH_{[s]}$ into $\cH_\bS$ that identifies $\cH_{[s]}$ with the
$s$-th component $\cH_{s} = \cH_{[s]}$
in the direct sum defining $\cH_{\bS}$ in \eqref{KSKR}:  $\iota_{s} h=
\oplus_{s'\in\bS} (\delta_{s'\!,s} h)$ for $h\in \cH_{[s]}$, with
$\delta_{s'\!,s}$ equal to the Kronecker delta.
 Similarly, for $r\in \bR$ we write $\iota_{r}$ for the
embedding of $\cH_{[r]}$ as the $r$-th component $\cH_{r} =
\cH_{[r]}$ in the direct-sum defining $\cH_\bR$ in \eqref{KSKR}.
Note that $\iota_s$ (respectively
$\iota_r$) acts on $\cH_{[s]}$
 (respectively $\cH_{[r]}$) and not on $\cH_{s}$ (respectively $\cH_{r}$), so
that for an $e\in E$ the product $\iota_{s(e)}\iota_{r(e)}^*$ is
properly defined.

We let $\bDelta^{\bE}$ denote the set of all operator-tuples
$Z = (Z_{e})_{e \in \bE}$ indexed by the
edge set $\bE$ such that the component $Z_{e}$ is in the $C^{*}$-algebra
$\bDelta_{[\br(e)]}=\bDelta_{[\bs(e)]}$.
Given any $Z = (Z_{e})_{e \in \bE} \in \bDelta^{\bE}$, we define an
operator $L_{\overline{\bG}}(Z) \in \cL(\cH_{\bR}, \cH_{\bS})$ by
\begin{equation}  \label{LG}
L_{\overline{\bG}}(Z) = \sum_{e \in \bE} \iota_{\bs(e)}
Z_{e} \iota_{\br(e)}^{*}.
\end{equation}
We then define the uncertainty set $\bDelta_{\overline{\bG}}$ associated with
the $A$-graph $\overline{\bG}$ by
\begin{equation}   \label{genstruc1}
    \bDelta_{\overline{\bG}} = \{ L_{\overline{\bG}}(Z) \colon
    Z = (Z_{e})_{e \in \bE}  \in \bDelta^{\bE}\}
     \subset \cL(\cH_{\bR}, \cH_{\bS}).
\end{equation}

Since the elements of $\bDe_\ov{\bG}$ in general are not square, we cannot work with
its commutant, like we did with $\bDe$ in \eqref{commutant}. Instead
we will make use of the intertwining space
\begin{equation}\label{InterStruc}
\bDe_\ov{\bG}'=\{(X,Y)\in \cL(\cH_\bR)\times \cL(\cH_\bS)\colon \De X= Y \De,\
\De\in\BDe_\ov{\bG}  \}.
\end{equation}
The following proposition gives an explicit description of this
intertwining space.

\begin{proposition}  \label{P:intertwining}
The set $\bDe_\ov{G}'$ is given by
\[
\bDe_\ov{G}'=\{(X,Y)\,\colon\, X=\sum_{r\in \bR}
\iota_{r}\Ga_{[r]}\iota_{r}^*,\, Y=\sum_{s\in \bS}
\iota_{s}\Ga_{[s]}\iota_{s}^* \text{ where } \Ga_p\in \bDe_{p}',\, p\in \bP\}.
\]
Here $\bDe_p'$ denotes the commutant of $\bDe_p$ in $\cL(\cH_p)$.
\end{proposition}

\begin{proof}
Assume the $C^*$-algebras $\bDe_p$, $p\in P$, are unital. If this is
not the case then one can modify the argument using approximate
identities. Let $(X,Y)\in\bDe_\ov{\bG}'$. Choose an $e_0\in \bE$ and take
$Z_{e_0}=I$ and $Z_{e'}=0$ for all $e'\not=e_0$. With this choice of
$Z=(Z_e)_{e\in \bE}\in \bDe^\bE$ the intertwining relation
$L_\ov{\bG}(Z)X=Y L_\ov{\bG}(Z)$ yields
\[
\iota_{\bs(e_0)}\iota_{\br(e_0)}^* X = Y \iota_{\bs(e_0)}\iota_{\br(e_0)}^*.
\]
Since $\iota_{v}^*\iota_{v}=I$  and $\iota_{v}^*\iota_{v'}=0$ for all
$v,v'\in V$ with $v\not= v'$ (and $v$ and $v'$ either both in $\bS$ or
both in $\bR$), we have
\[
\iota_{\br(e_0)}^* X\iota_{\br(e_0)} =\iota_{\bs(e_0)}^* Y \iota_{\bs(e_0)},\quad
\iota_{\br(e_0)}^* X\iota_{r} = 0\ (r\not= \br(e_0)), \quad
\iota_{s}^* Y \iota_{\bs(e_0)} = 0\ (s\not= \bs(e_0)).
\]
Set $X_r=\iota_{r}^* X\iota_{r}$ and $Y_s=\iota_{s}^* Y\iota_{s}$ for each
$r\in \bR$ and each $s\in \bS$. Since $e_0\in \bE$ was chosen arbitrarily, the above
identities
imply that
$$
    X  = \sum_{r,r' \in \bR} \iota_{r} \iota_{r}^{*} X \iota_{r'}
    \iota_{r'}^{*}  = \sum_{r \in \bR} \iota_{r} \iota_{r}^{*} X
    \iota_{r} \iota_{r}^{*} = \sum_{r \in \bR} \iota_{r} X_{r}
    \iota_{r}^{*}
$$
and similarly
$$
Y = \sum_{s,s' \in \bS} \iota_{s} \iota_{s}^{*} Y\iota_{s'}
\iota_{s'}^{*} = \sum_{s \in \bS} \iota_{s} \iota_{s}^{*} Y
\iota_{s} \iota_{s}^{*} = \sum_{s \in \bS} \iota_{s} Y_{s} \iota_{s}^{*}.
$$
Furthermore
$$
  X_{r} = X_{r'} = Y_{s} = Y_{s'} \text{ whenever } [r] = [r'] = [s]
  = [s'].
$$
We conclude that there is a well-defined operator $\Ga_{p}$ on
$\cH_{p}$ given by
$$
  \Ga_{p} = X_{r} = Y_{s} \text{ whenever } [r] = [s] = p
$$
and that $X$ and $Y$ are given by
\begin{equation}\label{XYform}
X=\sum_{r\in \bR} \iota_{r}\Ga_{[r]}\iota_{r}^*, \quad Y=\sum_{s\in \bS}
\iota_{s}\Ga_{[s]}\iota_{s}^*.
\end{equation}
 We show next that $\Ga_{p}\in
\bDe_p'$ for each $p$. Indeed, fix a $p\in \bP$
choose $\De_p\in\bDe_p$ and let $e_0\in \bE$ such that $[\bs(e)]=p$. We
take $Z=(Z_e)_{e\in E}\in\bDe^\bE$ with $Z_{e_0}=\De_p$ and
$Z_{e'}=0$ for $e'\not=e_0$. Then
$L_{\overline{\bG}}(Z)X=Y L_{\overline{\bG}}(Z)$ yields
\begin{align*}
\iota_{\bs(e_0)} \De_p \Ga_p \iota_{\br(e_0)}^*
&=\iota_{\bs(e_0)} \De_p \iota_{\br(e_0)}^*\iota_{\br(e_0)} \Ga_p \iota_{\br(e_0)}^*
=\iota_{\bs(e_0)} \De_p \iota_{\br(e_0)}^* \sum_{r\in \bR}\iota_{r} \Ga_{[r]}
\iota_{r}^*  \\
& =L_{\overline{\bG}}(Z)X=Y L_{\overline{\bG}}(Z)
=\sum_{s\in \bS}\iota_{s} \Ga_{[s]} \iota_{s}^* \iota_{\bs(e_0)} \De_p
\iota_{\br(e_0)}^* \\
& =\iota_{\bs(e_0)} \Ga_{[\bs(e_0)]} \iota_{\bs(e_0)}^* \iota_{\bs(e_0)} \De_p
\iota_{\br(e_0)}^*
=\iota_{\bs(e_0)} \Ga_{p}\De_p \iota_{\br(e_0)}^*.
\end{align*}
This proves that $\De_p \Ga_p=\Ga_{p}\De_p$. Since $\De_p$ is an
arbitrary element of $\bDe_p$ and $p\in\bP$ was also chosen arbitrarily, we obtain that
$\Ga\in \bDe_p'$ for each $p\in\bP$.

One easily verifies that the pair $(X,Y)$ with $X$ and $Y$ as in
\eqref{XYform} where $\Ga_p\in\bDe_p'$ for each $p\in \bP$ is in
$\bDe_\ov{G}'$. Hence the proof is complete.
\end{proof}

Now suppose that we are given an operator $M \in \cL(\cH_{\bS},
\cH_{\bR})$ along with the $A$-graph $\overline{\bG} = ( G, \{
\bDelta_{p} \subset \cL(\cH_{p})\}_{p \in \bP})$ as above.
We then define the {\em  $\mu_{\overline{\bG}}$-structured singular
value} of  $M$ as in \eqref{ssv}
but with  $\bDelta_{\overline{\bG}}$ as in \eqref{genstruc1} in place of $\bDelta$:
\begin{equation}\label{genmu}
\mu_{\bDelta_{\overline{\bG}}}(M) =
\frac{1}{ \inf \{ \| \Delta \| \colon \De \in \bDelta_{\overline{\bG}},\, 1 \in
\sigma(M \Delta) \} }.
\end{equation}
The analogue of the $D$-scaled version of $\mu$ is defined as
\begin{equation}  \label{genmuhat}
\widehat \mu_{\bDelta_{\overline{\bG}}}(M) = \inf\{ \| X M Y^{-1} \| \colon (X,Y) \in
\bDelta'_\bG \text{ with $X$, $Y$ invertible}\}.
\end{equation}
As in the classical case, $\widehat \mu_{\bDelta_{\overline{\bG}}}(M)$ has the
following properties:
\begin{itemize}
\item
Computation of $\widehat \mu_{\bDelta_{\overline{\bG}}}(M)$ can be reduced to a
$C^*$-algebra LOI (Linear Operator Inequality) computation: {\sl $\widehat
\mu_{\bDelta_{\overline{\bG}}}(M) < 1$ if and only if there
exists a positive definite structured solution $(X,Y) \in
\bDelta'_{\overline{\bG}}$
of the structured Stein inequality}
\begin{equation}   \label{LMItilde}
M^* X M - Y \prec 0.
\end{equation}
In cases of interest, the $C^{*}$-algebra is concretely identified as
a subspace of structured finite matrices and the structured LOI becomes a
structured LMI (Linear Matrix Inequality).

\item $\widehat \mu_{\bDelta_{\overline{\bG}}}(M)$ is always an
upper bound for $\mu_{\bDelta_{\overline{\bG}}}(M)$:\label{oneside}
\begin{equation}  \label{muGineq}
\mu_{\bDelta_{\overline{\bG}}}(M) \le \widehat \mu_{\bDelta_{\overline{\bG}}}(M).
\end{equation}
\end{itemize}

Rather than pursuing this general situation further, we now discuss
two particular special cases which will be our focus for the rest of
the paper.

\subsection{The classical uncertainty structure}  \label{S:classical}
Let us now suppose that we are given an M-graph $(G, \{ \cH_{p}\}_{p
\in \bP})$ and we take the $C^{*}$-subalgebra of $\cL(\cH_{p})$ to
be simply $\bDelta_{p} = \{ s I_{\cH_{p}} \colon s \in {\mathbb C}\}$.
Then a $Z \in \bDelta^{\bE}$ has the form $Z =
(Z_{e})_{e \in \bE}$ where $Z_{e} = \la_{e} I_{\cH_{p}}$ for
complex numbers $\la_{e}$.  Rather than write
$$
 L_{\overline{\bG}}(Z) = \sum_{e \in \bE} \iota_{\bs(e)} (\la_{e}
 I_{\cH_{[\br(e)]}} ) \iota_{\br(e)}^{*},
$$
we may write $L_{\overline{\bG}}(Z)$ directly as a
function of the tuple $(\la_{e})_{e \in \bE}$ of complex numbers:
\begin{equation}   \label{LGclassical}
L_{\overline{\bG}}(Z) = L_{\bG}(\la): = \sum_{e \in \bE} \la_{e}
L_{\bG,e}  \text{ where }
  L_{\bG,e} = \iota_{\bs(e)} \iota_{\br(e)}^{*} \text{ for } e \in
  \bE.
\end{equation}
Let us write more simply
\begin{equation}   \label{uncer-clas1}
\bDelta_{\bG} = \{ L_{\bG}(\la) \colon \la = (\la_{e})_{e \in \bE},\,
\la_{e} \in {\mathbb C}\}
\end{equation}
for the associated uncertainty structure
$\bDelta_{\overline{\bG}}$ with this special choice of
$C^{*}$-subalgebras $\bDelta_{p} = \{ s I_{\cH_{p}} \colon s \in {\mathbb
C}\}$.  Note next  that in this case $\bDelta_{p}' = \cL(\cH_{p})$.  We
therefore read off from Proposition \ref{P:intertwining} that the
intertwining space $\bDelta_{\bG}' : = \bDelta_{\overline{\bG}}'$ is
given by
\begin{equation}   \label{intertwine-classical}
\bDelta_{\bG}' : =
\{ (X, Y) \colon X = \sum_{r \in \bR} \iota_{r} \Gamma_{[r]}
\iota_{r}^{*}, \, Y = \sum_{s \in \bS} \iota_{s} \Gamma_{[s]}
\iota_{s}^{*} \text{ where } \Gamma_{p} \in \cL(\cH_{p}),\, p \in
\bP\}.
\end{equation}

To make $\bDelta_{\bG}$ more explicit, it is convenient to
introduce some auxiliary notation.  We let $\widetilde \cH_{s} =
{\mathbb C}$ for each source vertex $s \in \bS$ and similarly
$\widetilde \cH_{r} = {\mathbb C}$ for each range vertex $r \in \bR$.
For each connected component $p \in \bP$, we let
$$
\widetilde \cH_{\bS,p} = \bigoplus_{s \in \bS_{p}} \widetilde
\cH_{s}, \quad \widetilde \cH_{\bR,p} = \bigoplus_{r \in \bR_{p}}
\widetilde \cH_{r}
$$
and finally
$$
\widetilde \cH_{\bS} = \bigoplus_{p \in \bP} \widetilde \cH_{\bS,p},
\quad \widetilde \cH_{\bR} = \bigoplus_{p \in \bP} \widetilde
\cH_{\bR,p}.
$$
Note that these spaces amount to the quantities in \eqref{KSKR} in the case of the multiplicity-one assignment $ \cH_{p} = {\mathbb C}$ for each component $p$ of the graph $G$; in general we have the tensor
factorizations
\begin{equation}   \label{fact-clas}
    \cH_{\bR,p} = \widetilde \cH_{\bR,p} \otimes \cH_{p},
    \quad \cH_{\bS,p} = \widetilde \cH_{\bS,p} \otimes \cH_{p}.
\end{equation}
Then it is not difficult
to see that the uncertainty structure \eqref{uncer-clas1} can be
written more explicitly as
\begin{equation}   \label{uncer-clas2}
    \bDelta_{\bG} = \{ \bigoplus_{p \in \bP} W_{p} \otimes
    I_{\cH_{p}} \colon W_{p} \in \cL(\widetilde \cH_{\bR,p},
    \widetilde \cH_{\bS,p}) \}.
\end{equation}
Since $G$ is a finite graph, by assumption, we can number the path-connected components
$p_1,\ldots,p_K$,
with $K=\#(\bP)<\infty$. When convenient we shall use $k$ as an index rather than $p_k$
when referring to elements associated with the $k$-th connected component.
Say the $k$-th connected component $p_k$ has $n_{k}$
source vertices and $m_{k}$ range vertices. We then number the source vertices $s_{k,i}$ and
range vertices $r_{k,j}$
for $i=1\ldots,n_k$ and $j=1\ldots,m_k$ and write $e_{k,ij}$ for the edge connecting source
vertex $s_{k,i}$ to range
vertex $r_{k,j}$. Thus we have the following labelings:
 \begin{align*}
     & \bS = \cup_{k=1}^{K} \bS_{k} \text{ where }
\bS_{k} = \{s_{k,i} \colon 1 \le i \le n_{k}\}, \\
& \bR = \cup_{k=1}^{K} \bR_{k} \text{ where }
\bR_{k} = \{r_{k,j} \colon 1 \le j \le m_{k} \}, \\
& \bE = \cup_{k=1}^{K} \bE_{k} \text{ where }
\bE_{k} = \{ e_{k,ij} \colon 1 \le i \le n_{k},\, 1 \le j \le m_{k}\}.
\end{align*}
Then the uncertainty structure \eqref{uncer-clas1} now assumes the
form
$$
\bDelta_{\bG} = \{ \sum_{k,i,j} \la_{k,i,j} \iota_{s_{k,i}}
\iota_{r_{k,j}}^{*} \colon \la_{k,i,j} \in {\mathbb C} \text{
arbitrary}\}
$$
with the more explicit formulation \eqref{uncer-clas2} becoming
\begin{equation}   \label{uncer-clas3}
\bDelta_{\bG} = \{ {\rm diag}_{k=1, \dots, K} W_{k} \otimes
I_{\cH_{k}} \colon W_{k} \in {\mathbb C}^{n_{k} \times m_{k}} \}.
\end{equation}
In case all $\cH_k$ are finite dimensional, tensoring with $I_{\cH_{{k}}}$ just says that
each $\Delta_{k}$
is allowed to have multiplicity equal to $\dim \cH_{{k}}$.  We note that the structure
\eqref{structure} discussed
in Section \ref{S:intro} is the special case where $n_{k} = m_{k}$
for all $k$ and $\dim \cH_{{k}} = 1$ whenever $n_{k} = m_{k} > 1$.

\subsection{The enhanced classical uncertainty structure}
\label{S:enhanced}
We now describe a second special form for an A-graph.  Suppose that
we are given an M-graph $(G, \{ \cH_{p} \colon p \in \bP\}$ where
$\cH_{p}$ has the tensor-product form $\cH_{p} = \cK \otimes \cH^{\circ}_{p}$
for a fixed Hilbert space $\cK$ and coefficient Hilbert spaces
$\cH^{\circ}_{p}$.  It will be convenient to have a notation also for
the M-graph with coefficient Hilbert spaces $\cH^{\circ}_{p}$:
$$
   \bG^{\circ} = ( G, \{ \cH_{p}^{\circ} \colon p \in \bP\}).
$$
We now specify the $C^{*}$-subalgebra $\bDelta_{p} \subset
\cL(\cH_{p})$ to be
$$
  \bDelta_{p} = \cL(\cK) \otimes I_{\cH_{p}^{\circ}},
$$
and denote the associated A-graph by $\ov{\bG}$.
If $Z' = (Z'_{e})_{e \in \bE}$ is an element of
$\bDelta^{\bE}$, then each $Z'_{e}$ has the form
$$
 Z'_{e} = Z_{e} \otimes I_{\cH^{\circ}_{p}}
$$
where $Z_{e}$ is an arbitrary operator on $\cK$.  Then the operator
$$
L_{\overline{\bG}}(Z') = \sum_{e \in \bE} \iota_{\bs(e)} (Z_{e}
\otimes I_{\cH^{\circ}_{p}}) \iota_{\br(e)}^{*}
$$
is really a function $L_{\bG}(Z)$ of the $\bE$-tuple $Z = (Z_{e})_{e \in \bE}$ of
operators on $\cK$.  If we let $L_{\bG^{\circ}}(z)$ be as in Subsection
\ref{S:classical} associated with the M-graph $\bG^{\circ}$, with the $\circ$-super index
carried over in the notation,
$$
L_{\bG^{\circ}}(\la) = \sum_{e \in \bE} \la_{e} L_{\bG^{\circ},e}\quad \text{ where }\quad
L_{\bG^{\circ},e} = \iota^{\circ}_{\bs(e)} (\iota^{\circ }_{\br(e)})^{*},
$$
then, by using the identities
$$
\iota_{\bs(e)} = I_{\cK} \otimes
\iota^{\circ}_{\bs(e)}, \quad \iota_{\br(e)} = I_{\cK} \otimes
\iota^{\circ}_{\br(e)},
$$
it is easily verified that
\begin{equation}  \label{LGenhanced}
 L_{\overline{\bG}}(Z') = L_{\bG^{\circ}}(Z) : =
 \sum_{e \in \bE} Z_{e} \otimes L_{\bG^{\circ},e}.
\end{equation}
More explicitly, in the notation used at the end of Subsection
\ref{S:classical}, we see that we have the enhanced versions of the
factorizations \eqref{fact-clas}
\begin{equation}   \label{fact-en}
    \cH_{\bR,p} = \cK \otimes \cH^{\circ}_{\bR,p} = \cK \otimes \widetilde \cH_{\bR,p}
    \otimes
    \cH^{\circ}_{p}, \quad
 \cH_{\bS,p} = \cK \otimes \cH^{\circ}_{\bS,p} = \cK \otimes \widetilde \cH_{\bS,p} \otimes
 \cH^{\circ}_{p}
 \end{equation}
 and the associated uncertainty structure
$\bDelta_{\overline{\bG}}$ can be presented as follows:
\begin{equation}  \label{uncer-en2}
   \bDelta_{\overline{\bG}} = \left\{ \bigoplus_{p \in \bP} W_{p} \otimes
   I_{\cH^{\circ}_{p}} \colon
   W_{p} \in \cL(\cK \otimes \widetilde \cH_{\bR,p}, \cK \otimes
   \widetilde \cH_{\bS_{p}}) \right\}
 \end{equation}
 or in matrix form,
 \begin{equation}   \label{uncer-en3}
     \bDelta_{\ov{\bG}} = \{ W = {\rm diag}_{k=1, \dots, K} [ W_{k}
     \otimes I_{\cH^{\circ}_{p_{k}}} ] \colon W_{k} \in
     \cL(\cK)^{n_{k} \times m_{k}}\}.
\end{equation}

We shall be interested in computing $\mu_{\bDelta_{\overline{\bG}}}(M)$
for the case where $M$ has the tensored form $M = I_{\cK} \otimes M^{\circ}$
for an operator $M^{\circ} \in \cL(\cH_{\bS}^{\circ},
\cH_{\bR}^{\circ})$.  It is then natural to use the shorthand notation
$$
  \widetilde \mu_{\bG^\circ}(M^\circ) : = \mu_{\overline{\bG}}(I_{\cK} \otimes
  M^{\circ}).
$$

For $\bDelta_{p} = \cL(\cK) \otimes I_{\cH^{\circ}_{p}}$, we have
$$
  \bDelta_{p}' = I_{\cK} \otimes \cL(\cH^{\circ}_{p}).
$$
and hence we read off from Proposition \ref{P:intertwining} that
\begin{align}
    \bDelta_{\ov{\bG}}' & = \{ (X,Y) \colon X = \sum_{r \in \bR}
    I_{\cK} \otimes \iota_{r} \Gamma_{[r]}^{\circ} \iota_{r}^{*}, \,
    Y = \sum_{s \in \bS} I_{\cK} \otimes \iota_{s} \Gamma_{[s]}^{\circ}
    \iota_{s}^{*} \text{ where } \Gamma^{\circ}_{p} \in
    \cL(\cH^{\circ}_{p})\} \notag  \\
    & = I_{\cK} \otimes \bDelta_{\bG^{\circ}}'.
 \label{tildeDelta'}
 \end{align}

 \subsection{Main Result}   \label{S:main}
We can now state our Main Result as follows.

\smallskip

\begin{theorem}[Main Result]  \label{T:I}  Let $\overline{\bG}$
     and $\bG^{\circ}$ be as in Subsection \ref{S:enhanced} with
    $\cK$ taken to be an infinite-dimensional separable Hilbert and all $\cH_p^\circ$
    finite dimensional, where $p\in\bP$.
    Then, for any linear operator
\[
M^\circ:\bigoplus_{p \in \bP} \cH^{\circ}_{\bS,p} = \bigoplus_{p \in
    \bP} (\widetilde \cH_{\bS,p} \otimes
    \cH^{\circ}_{p})\to
\bigoplus_{p \in \bP} \cH^{\circ}_{\bR,p}
    = \bigoplus_{p \in \bP} (\widetilde
    \cH_{\bR,p} \otimes \cH^{\circ}_{p})
\]
we have
    $$
 \widetilde \mu_{\bDelta_{\bG^{\circ}}}(M^{\circ}): =
 \mu_{\bDelta_{\overline{\bG}}}(I_{\cK} \otimes M^{\circ}) = \widehat
 \mu_{\bDelta_{\bG^{\circ}}}(M^{\circ}).
    $$
In particular
$$
\widetilde \mu_{\bDelta_{\bG^{\circ}}}(M^{\circ})  < 1
\quad\Longleftrightarrow\quad \widehat \mu_{\bDelta_{\bG^{\circ}}}(M^\circ) < 1
$$
and testing whether $\widetilde \mu_{\bDelta_{\bG^{\circ}}}(M^{\circ}) < 1$ reduces to a
finite-dimensional LMI.
\end{theorem}

As explained in the Introduction, in the succeeding sections we
discuss two distinct approaches to this result: one based on the
earlier work of Ball-Groenewald-Malakorn \cite{BGM3}, the other on the work of
Dullerud-Paganini \cite{Paganini, DP}.

We conclude this section with a remark that reduces the claims of Theorem \ref{T:I} to a single implication.

\begin{remark}\label{R:reduct}  {\em
We first observe that the inequality $\wtil{\mu}_{\bDe_{\bG^\circ}}(M^\circ)\leq \widehat{\mu}_{\bDe_{\bG^\circ}}(M^\circ)$ holds. This follows from two observations. Firstly, we have the inequality $\wtil{\mu}_{\bDe_{\bG^\circ}}(M^\circ) =\mu_{\bDelta_{\overline{\bG}}}(I_{\cK} \otimes M^{\circ}) \leq \widehat{\mu}_{\bDe_{\overline{\bG}}}(I_\cK\otimes M^\circ)$, as observed on the level of Subsection \ref{S:uncertainty} on Page \pageref{oneside}.
Secondly, since
$\bDe_{\overline{\bG}}'=I_{\cK}\otimes \bDe_{\bG^\circ}'$,
by \eqref{tildeDelta'}, we have
\[
\|X(I_\cK\otimes M^\circ)Y^{-1}\|=\|X^\circ M^\circ (Y^\circ)^{-1}\|
\]
for any $(X,Y)=(I_\cK\otimes X^\circ, I_\cK\otimes Y^\circ)\in \bDe_{\overline{\bG}}'$ with $(X^\circ,Y^\circ)\in \bDe_{\bG^\circ}'$. Consequently, we obtain $\widehat{\mu}_{\bDe_{\overline{\bG}}}(I_\cK\otimes M^\circ)= \widehat{\mu}_{\bDe_{\bG^\circ}}(M^\circ)$, which yields the claimed inequality. Hence it remains to prove $\widehat{\mu}_{\bDe_{\bG^\circ}}(M^\circ) \leq \wtil{\mu}_{\bDe_{\bG^\circ}}(M^\circ)$.
By a scaling argument, this in turn reduces to
showing:
\begin{equation}   \label{reduction2}
    \widetilde \mu_{\bG^\circ}(M^\circ) < 1\quad \Longrightarrow \quad \widehat
    \mu_{\bG^\circ}(M^\circ) < 1.
\end{equation}
}\end{remark}

\section{Noncommutative Bounded Real Lemma, State-Space Similarity Theorem, and
structured singular value versus diagonal scaling}  \label{S:BRLproof}

Throughout this section, let $\bG$ be a M-graph:
$$
  \bG = (G, \{\cH_{p} \colon p \in \bP\}).
$$
Here we give a proof of our Main Result (Theorem \ref{T:I}) based on two theorems
from \cite{BGM2,BGM3} regarding the
Schur-Agler class and colligation matrices associated with the M-graph $\bG$.

For this purpose we let $z=(z_{e})_{e \in \bE}$ be a collection of freely
noncommuting indeterminates indexed by the edge set $\bE$.
We let
$L_{\bG}(z)$ be the formal linear pencil
\begin{equation}   \label{formalpencil}
    L_{\bG}(z): = \sum_{e \in \bE} z_{e}
L_{\bG,e}
\end{equation}
where the coefficients $L_{\bG,e}$ are as in \eqref{LGclassical}.
For $Z = (Z_{e})_{e \in \bE}$ a tuple of operators on some auxiliary
Hilbert space $\cK$, we evaluate the formal pencil $L_{\bG}(z)$ at
the argument $Z$ by using tensor products just as in
\eqref{LGenhanced}:
\begin{equation}   \label{LGenhanced'}
    L_{\bG}(Z) = \sum_{e \in \bE} Z_{e} \otimes L_{\bG,e}.
 \end{equation}
 This framework includes as a special case the situation where $\cK =
{\mathbb C}$ and each $Z_{e}$ is an operator on the one-dimensional
space ${\mathbb C}$; for this case we write $\lambda = (\lambda_{e})_{e\in \bE}
$ with $\la_e\in\BC$ instead of $Z = (Z_{e})_{e\in\bE}$ and we arrive at the classical
operator pencil in the $\bE$-tuple of
complex numbers $\lambda = (\lambda_{e})_{e\in \bE}$ as in
\eqref{LGclassical}:
$$
  L_{\bG}(\lambda) = \sum_{e \in \bE} \la_{e}\, L_{\bG,e}.
$$

Before turning to the results from \cite{BGM2,BGM3}
and the proof of Theorem \ref{T:I}, we recall some facts about formal power
series.

\subsection{Formal power series.}\label{S:FPS}

We let $\cF_{\bE}$ be the free monoid on the generating set
$\bE$, i.e., the free semigroup with the empty word $\emptyset$ serving as the identity
element.  Thus a generic element $\alpha$ of $\cF_{\bE}$
has the form $\alpha = e_{i_{N}} \cdots e_{i_{1}}$ where $e_{i_{j}}
\in \bE$ for each $j =1, \dots, N$.  When $\alpha \in \cF_{\bE}$ has this form, we
say that the {\em length} $|\alpha|$ of
$\alpha$ is
$N$; we include the empty word $\emptyset$ as an element of
$\cF_{\bE}$, considered to have length zero.  Multiplication
of two elements $\alpha  = e_{i_{N}} \cdots e_{i_{1}}$ and $\beta =
e_{j_{M}} \cdots e_{j_{1}}$ of $\cF_{\bE}$ is by
concatenation:
$$
\alpha \cdot \beta = e_{i_{i_{N}}} \cdots e_{i_{1}} e_{\beta_{j_{M}}}
\cdots e_{\beta_{j_{1}}}
$$
with the empty word $\emptyset$ serving as the identity element of
$\cF_{\bE}$. Furthermore, the transpose $\al^\top$ of
$\alpha  = e_{i_{N}} \cdots e_{i_{1}}$ is defined as $\al^{\top} = e_{j_{1}}
\cdots e_{j_{M}}$.
Given the $\bE$-tuple $z = (z_{e})_{e\in\bE}$ of freely noncommuting indeterminates and
an element $\alpha = e_{i_{N}} \cdots e_{i_{1}}$ we define the noncommutative monomial $z^{\alpha}$ by
$$
z^{\alpha} = z_{e_{i_{N}}} \cdots z_{e_{i_{1}}}
$$
with an individual indeterminate $z_{e}$ identified with $z^{\alpha}$
if $\alpha  = e$ is a word of length one and with $z^{\emptyset}$
identified with $1$.

For $\cX$ a linear space, we let $\cX\langle
\langle z \rangle \rangle$ denote the set of all formal power series
$ \sum_{\alpha \in \cF_{\bE}} x_{\alpha} \, z^{\alpha}$
with coefficients $x_{\alpha}$ coming from $\cX$. Two formal power series
$ \sum_{\alpha \in \cF_{\bE}} x_{\alpha} \, z^{\alpha}$
and $ \sum_{\alpha \in \cF_{\bE}} y_{\alpha} \, z^{\alpha}$ are said to be
equal if $x_\al=y_\al$ for all $\al\in \bE$. If $\cX'$
and $\cX''$  are also linear spaces for which a multiplication $\cX'
\times \cX \to \cX''$ is defined and if we are given two formal series
$x(z) = \sum_{\alpha\in\cF_\bE} x_{\alpha} z^{\alpha} \in \cX\langle
\langle z \rangle \rangle$ and $x'(z) = \sum_{\beta\in\cF_\bE}
x'_{\beta} z^{\beta} \in \cX'\langle
\langle z \rangle \rangle$, then the product formal series $x'(z) \cdot
x(z) \in \cX''\langle
\langle z \rangle \rangle$ is always well defined and given by
$$
  (x' \cdot x)(z) = \sum_{ \gamma \in \cF_{\bE}}
  \left( \sum_{ \beta, \alpha \in \cF_{\bE} \colon  \beta
  \cdot \alpha = \gamma} x'_{\beta} x_{\alpha} \right) z^{\gamma}.
$$

Assume $\cX$ is endowed with some appropriate topology (typically $\cX$ will be a Hilbert
space or the space of bounded linear operators between two Hilbert
spaces). As is now common in the theory of noncommutative functions (see
e.g.~\cite{HKMcC2011, K-VVbook}), we will often view a formal power series
$x(z) = \sum_{\alpha\in\cF_\bE} x_{\alpha} z^{\alpha} \in \cX\langle
\langle z \rangle \rangle$ as a function whose variables are operators on some auxiliary
separable Hilbert space $\cK$.
In this way, for an $\bE$-tuple $Z = (Z_{e})_{e\in\bE}$ of
linear operators acting on
$\cK$ and a formal power series $x(z) = \sum_{\alpha
\in \cF_{\bE}} x_{\alpha}\, z^{\alpha}$  we define an
element $x(Z) \in \cL(\cK)\otimes \cX $ by
\begin{equation}\label{OperatorEval}
x(Z) = \sum_{\alpha \in \cF_{\bE}} Z^{\alpha} \otimes x_{\alpha}
  \in \cL(\cK) \otimes \cX
\end{equation}
whenever the series converges in the appropriate topology of $\cL(\cK)\otimes \cX$. Here we
use the notation
$$
 Z^{\alpha} = Z_{e_{i_{N}}} \cdots Z_{e_{i_{1}}} \in  \cL(\cK)\quad\mbox{for
 $\alpha  = e_{i_{N}} \cdots e_{i_{1}}\in\cF_\bE$}.
$$
Notice that the point evaluation in \eqref{OperatorEval} generalizes
the one already introduced for the linear case in \eqref{LGenhanced}.

\subsection{The Schur-Agler class and colligation matrices associated
with $\bG$.} \label{S:SchurAgler}

Let $\cU$ and $\cY$ be two auxiliary Hilbert spaces. Given a formal
power series $S(z) = \sum_{\alpha \in \cF_{\bE}}
S_{\alpha} z^{\alpha} \in \cL(\cU, \cY)\langle \langle z \rangle
\rangle$, we say that $S$ is in the Schur-Agler class
$\mathcal{SA}_{\bG}(\cU,\cY)$ associated with the M-graph $\bG$ if
for any $\bE$-tuple $Z = (Z_{e})_{ e \in \bE}$ of operators $Z_{e} \in \cL(\cK)$
such that $\| L_{\bG}(Z) \| <1$, the evaluation $S(Z)$ via \eqref{OperatorEval} is in
$\cL(\cK \otimes \cU , \cK \otimes \cY )$ and satisfies
$ \| S(Z) \| \le 1$.
We note that the test-class of $\bE$-tuples $Z = (Z_{e})_{e\in\bE}$ is independent of
the choice of multiplicity
structure for $\bG$, as changing the multiplicity structure of $\bG$
does not effect the norm $\| L_{\bG}(Z) \|$. For purposes of
defining the Schur-Agler class, we may as well assume that the
underlying graph $G$ is taken with multiplicity-1 structure ($\cH_{p}
= {\mathbb C}$ for each $p$), and we
write $\mathcal{SA}_{G}(\cU, \cY)$ rather than
$\mathcal{SA}_{\bG}(\cU, \cY)$.

The following result was obtained in \cite{BGM2}

\begin{theorem}\label{T:BGM2} (See \cite[Theorem 5.3]{BGM2}.)
Give
    a formal power series $S \in \cL(\cU, \cY)\langle \langle z
    \rangle \rangle$, $S(z) = \sum_{\alpha \in \cF_{\bE}}
S_{\alpha} z^{\alpha}$,
the following conditions are equivalent:
    \begin{enumerate}
	\item $S$ is in the Schur-Agler class
	$\mathcal{SA}_{G}(\cU, \cY)$.
	
	\item
There is a multiplicity assignment $\{\cH_{p} \colon p\in \bP\}$
 giving rise to
	 an M-graph  $\bG = (G, \{\cH_{p}  \colon p \in \bP\})$ and a formal power
	 series $H \in \cL(\cH_{\bS}, \cY)\langle \langle z
    \rangle \rangle$ so that
$S$ has the {\em Agler decomposition}
\begin{equation}   \label{Aglerdecom}
  I - S(z) S(w)^{*} = H(z)  (I - L_{\bG}(z) L_{\bG}(w)^{*})  H(w)^{*}.
\end{equation}
Here $\overline{w} = (\overline{w}_e)_{e\in\bE}$ is another $\bE$-tuple of
freely noncommuting indeterminates,
we set $H(w)^{*} =\sum_{\beta \in \cF_{\bE}} (H_{\beta})^{*}
\overline{w}^{\beta^{\top}}$ if $H(z) = \sum_{\alpha \in {\cF}_{\bE}} H_{\alpha}
z^{\alpha}$ and define $S(w)^{*}$ accordingly, and \eqref{Aglerdecom} is to be interpreted
as an formal power series in the
$\bE \dot\cup\bE$-tuple $(z_e)_{e\in\bE}\cup (\overline{w}_e)_{e\in\bE}$.

\item $S$ has a {\em dissipative noncommutative structured
realization}, i.e., there exists a multiplicity assignment $\{\cH_{p}
\colon p \in \bP\}$
with associated $M$-graph
$$
\bG = (G, \{ \cH_{p} \colon p \in \bP \})
$$
together with a contractive colligation matrix
\begin{equation}   \label{col}
\bU =  \begin{bmatrix} A & B \\ C & D \end{bmatrix}  \colon
\begin{bmatrix} \cH_{\bS}  \\ \cU \end{bmatrix}
    \to \begin{bmatrix} \cH_{\bR}  \\ \cY
\end{bmatrix}
\end{equation}
so that
\begin{equation}   \label{real}
  S(z) = D + C (I - L_{\bG}(z) A)^{-1} L_{\bG}(z) B.
\end{equation}
\end{enumerate}
\end{theorem}

If we are given a colligation matrix $\bU$ as in \eqref{col} and define the
associated formal power series $S(z)$ via \eqref{real}, then it is
possible that $S$ is in the Schur-Agler class even though the
colligation matrix $\bU$ is not contractive; indeed, a sufficient
condition which is weaker than contractivity of $\bU$ is that there exist an
invertible change-of-basis matrix $\Gamma_{p}$ on $\cH_{p}$ for each connected
component $p \in \bP$ of $\bG$ so that the transformed
colligation matrix
$$
\bU' = \left[\begin{matrix} A' & B' \\ C' & D'
\end{matrix}\right] : =
\left[\begin{matrix} \bigoplus_{r \in \bR} \Gamma_{[r]} & 0 \\ 0 & I
\end{matrix}\right]  \left[ \begin{matrix} A & B \\ C & D
\end{matrix} \right]
\left[ \begin{matrix} \bigoplus_{s \in \bS} (\Gamma_{[s]})^{-1} & 0 \\ 0 & I
\end{matrix} \right]
$$
is a contraction:
\begin{equation}   \label{U'con1}
  \left\| \begin{bmatrix} \bigoplus_{r \in \bR} \Gamma_{[r]} & 0 \\ 0 & I
\end{bmatrix}  \begin{bmatrix} A & B \\ C & D \end{bmatrix}
\begin{bmatrix} \bigoplus_{s \in \bS} (\Gamma_{[s]})^{-1} & 0 \\ 0 & I
\end{bmatrix} \right\| \le 1.
\end{equation}
Equivalently, one can ask for
positive definite matrices $\Ga_{p} \succ 0$ on each partial state space
$\cH_{p}$ so that
\begin{equation}   \label{U'con2}
 \begin{bmatrix} A^{*} & C^{*} \\ B^{*} & D^{*} \end{bmatrix}
 \begin{bmatrix} \bigoplus_{r \in \bR} \Gamma_{[r]} & 0 \\ 0 & I
 \end{bmatrix} \begin{bmatrix} A & B \\ C & D \end{bmatrix} -
 \begin{bmatrix} \bigoplus_{s \in \bS} \Gamma_{[s]} & 0 \\ 0 & I
 \end{bmatrix} \preceq   0.
\end{equation}
If we assume that all the spaces $\cH_{p}$ are finite-dimensional and
also impose a structured minimality assumption, this sufficient
condition is also necessary (see Theorem 3.1 in \cite{BGM3}).
A result of this type is known as a {\em Bounded Real Lemma} (see
e.g.~\cite{ZDG}).  The idea of a {\em strict Bounded Real Lemma} (see
e.g.~\cite{PAJ} and Lemma 7.4 in \cite{DP})  is to trade in the
minimality assumption for a stability assumption.  The Bounded Real Lemma in the
context of SNMLSs is the following result.

\begin{theorem}  \label{T:strictBRL}  (See \cite[Theorem 3.4]{BGM3}.)
    Suppose that we are given an A-graph of the form $\bG =
    (G, \{ \bDelta_{p} = \{ s I_{\cH_{p}} \colon s \in {\mathbb C}\}
    \subset \cL(\cH_{p}) \})$, where $\cH_{p}$ is a finite-dimensional Hilbert space for
    each $p \in \bP$, together with  a
    colligation matrix $\bU$ as in \eqref{col}.  Associate with $\bU$
    the formal power series $S(z)$ as in \eqref{real}.
    Then the following conditions are equivalent:
\begin{enumerate}
    \item (i) $A$ is uniformly $\bG$-stable:
    $$
    \sup_{Z \colon \| L_{G}(Z)\| \le 1}\| (I - L_{\bG}(Z) A)^{-1} \|
    < \infty
    $$
and  (ii) there exists a $\rho < 1$ so that $S \in \rho \cdot
\mathcal{SA}_{G}(\cU, \cY)$:
$$
   \sup_{Z \colon \| L_{\bG}(Z) \| \le 1}  \| S(Z) \| \le \rho < 1.
$$

\item There exist invertible matrices $\Gamma_{p}$ on $\cH_{p}$, for each $p \in \bP$, so
that the strict version of condition \eqref{U'con1} holds:
\begin{equation}   \label{strictU'con1}
      \left\| \begin{bmatrix} \bigoplus_{r \in \bR} \Gamma_{[r]} & 0 \\ 0 & I
\end{bmatrix}  \begin{bmatrix} A & B \\ C & D \end{bmatrix}
\begin{bmatrix} \bigoplus_{s \in \bS} (\Gamma_{[s]})^{-1} & 0 \\ 0 & I
\end{bmatrix} \right\| < 1.
\end{equation}

\item There exist strictly positive definite operators $\Gamma_{p}$
on  $\cH_{p}$, for each $p \in \bP$, so that the strict version of
\eqref{U'con2} holds:
\begin{equation}   \label{strictU'con2}
     \begin{bmatrix} A^{*} & C^{*} \\ B^{*} & D^{*} \end{bmatrix}
 \begin{bmatrix} \bigoplus_{r \in \bR} \Gamma_{[r]} & 0 \\ 0 & I
 \end{bmatrix} \begin{bmatrix} A & B \\ C & D \end{bmatrix} -
 \begin{bmatrix} \bigoplus_{s \in \bS} \Gamma_{[s]} & 0 \\ 0 & I
 \end{bmatrix} \prec   0.
\end{equation}
 \end{enumerate}
\end{theorem}


\subsection{Proof of Theorem \ref{T:I}.}\label{S:MainProof1}

For the remainder of this section we follow the notation of Subsections \ref{S:enhanced}
and \ref{S:main}. Hence,
we consider an M-graph $(G,\{\cH_p\colon p\in \bP\})$ where each Hilbert space $\cH_p$ has
the tensored form
$\cH_p=\cK\oplus\cH_p^\circ$ with $\cK$ and $\cH_p^\circ$ Hilbert spaces, $\cK$ separable
and $\cH_p^\circ$ finite dimensional.
As in Subsection~\ref{S:enhanced}, with this M-graph we associate the M-graph
$\bG^\circ=(G,\{\cH_p^\circ \colon p\in \bP\})$ and the A-graph
\[
\ov{\bG}=(G,\{\bDe_p=\cL(\cK)\otimes I_{\cH_p^\circ}\subset\cL(\cH_p),\ p\in\bP\}).
\]
The linear pencils $L_\bG(\la)$ and $L_\bG(Z)$ from Subsection \ref{S:SchurAgler} then
coincide with $L_{\bG^\circ}(\la)$ and
$L_{\bG^\circ}(Z)=L_{\ov{\bG}}(Z')$, respectively, as defined in Subsection \ref{S:enhanced}.
We proceed here with the notation of Subsection \ref{S:enhanced}, i.e.,
with $L_{\bG^\circ}(\la)$ and $L_{\bG^\circ}(Z)$, as well as the
formal pencil $L_{\bG^{\circ}}(z)$ as in \eqref{LGenhanced'}.

Now let us suppose we are given a matrix $M^\circ\in \cL(\cH_{\bS}^\circ, \cH_{\bR}^\circ)$,
where
\begin{align*}
\cH_\bS^\circ&=\bigoplus_{p\in\bP} \wtil{\cH}_{\bS,p}\otimes\cH_p^\circ
=\bigoplus_{p\in\bP} (\oplus_{s\in\bS} \wtil{\cH}_s \otimes\cH_p^\circ),\\
\cH_\bR^\circ&=\bigoplus_{p\in\bP} \wtil{\cH}_{\bR,p}\otimes\cH_p^\circ
=\bigoplus_{p\in\bP} (\oplus_{r\in\bR} \wtil{\cH}_r \otimes\cH_p^\circ),
\end{align*}
where $\wtil{\cH}_s=\wtil{\cH}_r=\BC$ for each $s\in\bS,\, r\in\bR$. As before we set
$M=I_{\cK}\otimes M^\circ\in \cL(\cK\otimes\cH_\bR^\circ,\cK\otimes\cH_\bS^\circ)$.

For the discussion to follow let us introduce the notation
$$
  \overline{\cB}\bDelta_{\bG^\circ}
  = \{ L_{\bG^\circ}(Z) \colon Z = (Z_{e})_{e\in\bE}, Z_e \in \cL(\cK) \text{ with }
  \| L_{\bG^\circ}(Z)\| \le 1\}.
$$
As observed in Remark \ref{R:reduct}, it remains to prove the implication:
\begin{equation}   \label{reduction}
    \widetilde \mu_{\bG^\circ}(M^\circ) < 1\quad \Longrightarrow \quad \widehat
    \mu_{\bG^\circ}(M^\circ) < 1.
\end{equation}

The assumption $\widetilde \mu_{\bG^\circ}(M^\circ) < 1$ implies in
particular that
\begin{equation}  \label{Mstable}
    (I - L_{\bG^\circ}(Z)M)^{-1} \text{ exists for all } Z \text{ with }
    \|L_{\bG^\circ}(Z)\| \le 1.
\end{equation}
We note that the formal structured resolvent $(I - L_{\bG^\circ}(z) M^\circ)^{-1}$ can be
written in realization form \eqref{real}
\begin{equation}   \label{resolvent}
    (I - L_{\bG^\circ}(z) M^\circ)^{-1} = I + I \cdot (I - L_{\bG^\circ}(z)
    M^\circ)^{-1} L_{\bG^\circ}(z) \cdot M^\circ,
\end{equation}
i.e., in the form \eqref{real} with $\left[ \begin{smallmatrix} A & B \\
C & D \end{smallmatrix} \right] = \left[ \begin{smallmatrix} M^\circ & M^\circ \\
I & I \end{smallmatrix} \right]$.  If condition \eqref{Mstable} can
be strengthened to
\begin{equation}   \label{unifMstable}
    \sup_{Z \in \cL(\cK)^{\bE} \colon \| L_{\bG^\circ}(Z) \| \le 1} \| (I - L_{\bG^\circ}(Z) M)^{-1}
    \| < \infty
\end{equation}
then condition (i) in statement (1) of Theorem \ref{T:strictBRL} (with $\bG$ replaced by
$\bG^\circ$) is satisfied with $M^\circ$ in place of $A$.  Moreover, if \eqref{unifMstable}
holds and if we chose a positive number $r$  slightly larger than the
supremum in \eqref{unifMstable}, then the power series $S(z) =
\frac{1}{r} \cdot (I - L_{\bG^\circ}(z) M^\circ)^{-1}$ meets condition (ii) in
statement (1) of Theorem \ref{T:strictBRL}.  From the formula
\eqref{resolvent} we see that this $S(z)$ has a realization
\eqref{real} with
$$
   \begin{bmatrix} A & B \\ C & D \end{bmatrix} = \begin{bmatrix} M^\circ &
       M^\circ \\ \frac{1}{r} I & \frac{1}{r} I \end{bmatrix}.
$$

We may then use the implication (1) $\Rightarrow$ (3) in Theorem
\ref{T:strictBRL} to conclude that there exist strictly positive
definite $\Gamma_{p} \succ 0$  on $\cH_{p}$ ($p \in \bP$) so that
$$
 \begin{bmatrix} (M^\circ)^{*} & \frac{1}{r} I \\ (M^\circ)^{*} & \frac{1}{r} I
 \end{bmatrix} \begin{bmatrix} \bigoplus_{r \in \bR} \Gamma_{[r]} &
 0 \\ 0 & I \end{bmatrix} \begin{bmatrix} M^\circ & M^\circ \\ \frac{1}{r} I &
 \frac{1}{r} \end{bmatrix} - \begin{bmatrix} \bigoplus_{s \in \bS}
 \Gamma_{[s]} & 0 \\ 0 & I \end{bmatrix} \prec 0.
$$
In particular, peeling off the $(1,1)$-entry in this block-matrix
inequality yields
$$
  (M^\circ)^{*} \left( \bigoplus_{r \in \bR} \Gamma_{[r]}\right) M^\circ -
  \bigoplus_{s \in \bS} \Gamma_{[s]} \prec 0
$$
from which we read off that $\widehat \mu_{\bG^\circ}(M^\circ) < 1$ as required.
This analysis completes a proof of Theorem \ref{T:I} pending a justification for
the jump from \eqref{Mstable} to \eqref{unifMstable}.

We note that without loss of generality we may take the separable
infinite-dimensional Hilbert space $\cK$ to be $\ell^{2}$
(the space of square-summable complex-valued sequences indexed by the nonnegative
integers ${\mathbb Z}_{+}$).
We conclude that the following lemma, when specialized to the case $M =
I_{\ell^{2}} \otimes M^{\circ}$ and combined with the analysis in the
previous discussion,  leads to a complete proof of Theorem \ref{T:I}.
The construction of the key operator $\widehat W$ in the proof adapts
ideas from the proof of Proposition B.1 in \cite{DP} which can be
traced further back to the work of Shamma \cite{Shamma}.

\begin{lemma}  \label{L:DP}
Let $M \in \cL(\ell^{2} \otimes \cH^{\circ}_{\bS},
    \ell^{2} \otimes \cH^{\circ}_{\bR})$ be shift invariant:
$M V_{\bS} = V_{\bR} M$ where we set $V_{\bR} = V \otimes
I_{\cH^{\circ}_{\bR}}$,
$V_{\bS} = V \otimes I_{\cH^{\circ}_{\bS}}$ where $V$ is the
unilateral shift operator on $\ell^{2}$:
$$
V \colon (x_{0}, x_{1}, \dots) \mapsto (0, x_{0}, x_{1}, \dots).
$$
Assume that the inverse $(I - L_{{\bG^circ}}(Z) M)^{-1}$ exists for all $\bE$-tuples  $Z= (Z_{e})_{e\in\bE}$ in $\cL(\ell^2)$ such that $\| L_{\bG^\circ}(Z) \| \le 1$.  Then the collection of all such inverses is uniformly bounded:
\begin{equation}   \label{sup}
\sup \{ \| (I - L_{{\bG^\circ}}(Z) M)^{-1} \| \colon Z= (Z_{e})_{e\in\bE},\, Z_e\in\cL(\ell^2) \text{ with }
\| L_{{\bG^\circ}}(Z)\| \le 1\} <
\infty.
\end{equation}
 \end{lemma}

\begin{proof}
For integers $0\leq n_0\leq N$, let $\ell^{2}[n_0,N]$ denote the subspace of sequences in $\ell^2$ with support in the positions indexed by $n_0,\ldots, N$; similarly $\ell^{2}[n_0,N)$ and $\ell^{2}[n_0,\infty)$ stand for the subspaces $\ell^2$ with support in $n_0,\ldots,N-1$ and $n_0,\ldots$.  As a matter of notation we write $P_{[n_{0},N]}$ for the
orthogonal projection of $\ell^{2} \otimes \cX$ onto $\ell^{2}[n_0,N]
\otimes \cX$  (where $\cX$ is either $\cH^{\circ}_{\bS}$ or
$\cH^{\circ}_{\bR}$ depending on the context); when $n_{0} = 0$ we
write more simply $P_{N}$ rather than $P_{[0,N]}$.

We proceed by contradiction.
     Suppose that $(I - \De M)^{-1}$ exists for all $\De \in
     \overline{\cB}\bDelta_{{\bG^\circ}}$ but that the supremum in
     \eqref{sup} is infinite.  Fix any sequence of positive numbers
     $\epsilon_{n} > 0$ such that $\lim_{n \to \infty} \epsilon_{n} =
     0$. Then we can find $\De^{(n)}\in \overline{\cB}\bDelta_{{\bG^\circ}}$, i.e.,
\begin{equation}\label{DenStruc}
\De^{(n)} = {\rm diag}_{p \in \bP}
     W_{p}^{(n)} \otimes I_{\cH^{\circ}_{p}}\quad \mbox{with}\quad \| \De^{(n)}\| \le
     1,
\end{equation}
along with unit vectors $q^{(n)} \in \ell^{2} \otimes
     \cH^{\circ}_{\bS} = \bigoplus_{p \in \bP} \ell^{2} \otimes
     \widetilde \cH_{\bS_{p}} \otimes \cH^{\circ}_{p}$
     so that
     \begin{equation}   \label{est1}
\| (I - \De^{(n)} M) q^{(n)} \| < \epsilon_{n}.
\end{equation}

Observe that then, for any $n_{0} \in {\mathbb Z}_{+}$,
\begin{align*}
    \epsilon_{n} & > \| (I - \De^{(n)}M) q^{(n)}\| \\
    & = \| V_{\bS}^{n_{0}} (I - \De^{(n)} M)
    q^{(n)} \| \\
& = \| V_{\bS}^{n_{0}}  q^{(n)} - V_{\bS}^{n_{0}} \De^{(n)}V_{\bR}^{* n_{0}} V_{\bR}^{n_{0}}
M q^{(n)} \| \\
 & =  \| V_{\bS}^{n_{0}}  q^{(n)} -
    \widetilde \De^{(n)} M  V_{\bS}^{n_{0}} q^{(n)} \| \\
    & = \| (I - \widetilde \De^{(n)} M) \widetilde q^{(n)} \|
\end{align*}
where we have set
$$
\widetilde  \De^{(n)} =  V_{\bS}^{n_{0}} \De^{(n)}
    V_{\bR}^{* n_{0}}, \quad \widetilde
    q^{(n)} =   V_{\bS}^{n_{0}} q^{(n)}
$$
and we used the assumed shift-invariance property $V_{\bR} M = M V_{\bS}$ of $M$.
Using the representation of $\overline{\cB}\bDelta_{{\bG^\circ}}$ in \eqref{uncer-en2}, it follows that $\widetilde \De^{(n)}$ is in $\overline{\cB}\bDelta_{{\bG^\circ}}$, since $\De^{(n)}$ is in $\overline{\cB}\bDelta_{{\bG^\circ}}$ (see \eqref{DenStruc}). Moreover, $\widetilde q^{(n)}$ is again a unit vector,
but now with support in $[n_{0}, \infty)$. Also $\widetilde \De^{n}$
maps $\ell^{2}([n_{0}, \infty)) \otimes \cH^{\circ}_{\bR}$ into
$\ell^{2}([n_{0}, \infty)) \otimes \cH^{\circ}_{\bS}$.
We conclude that without
loss of generality we may assume that \eqref{est1} holds with the
additional normalization that the unit vector $q^{(n)}$ has support
in $[n_{0}, \infty)$ and $\De^{(n)} \in
\overline{\cB}\bDelta_{{\bG^\circ}} \cap \cL(\ell^{2}[n_{0},
\infty) \otimes \cH^{\circ}_{\bR}, \ell^{2}[n_{0},
\infty) \otimes \cH^{\circ}_{\bS})$ where $n_{0}$ is any nonnegative integer of our
choosing.

 A familiar fact
is that $P_{N} \to I$ strongly as $N \to \infty$.  We now develop
several consequences of this observation.

From the identity
\begin{align*}
& (I - P_{N} \De^{(n)} M) P_{N} q^{(n)} =  \\
&\qquad (I - \De^{(n)} M) P_{N} q^{(n)} +
(I - P_{N}) \De^{(n)} M q^{(n)} - (I - P_{N}) \De^{(n)} M (I - P_{N})
q^{(n)}
\end{align*}
we get the estimate
\begin{align*}
   & \| (I - P_{N} \De^{(n)} M) P_{N} q^{(n)}\|  \le
    \| (I - \De^{(n)} M)P_{N} q^{(n)} \| + \| (I - P_{N}) \De^{(n)} M
    q^{(n)} \|  \\
    & \qquad\qquad + \| (I - P_{N}) \De^{(n)} M (I - P_{N}) q^{(n)} \| \\
    & \quad\qquad\le \| (I - \De^{(n)} M) P_{N} q^{(n)} \| + \| (I - P_{N})
    \De^{(n)} M q^{(n)} \| + \| M \| \| (I - P_{N}) q^{(n)} \|.
\end{align*}
By the strong convergence of $\{P_{N}\}$ to the identity operator, the last two terms of the final
expression tend to 0 as $N \to \infty$.  We arrive at the estimate
\begin{equation}  \label{est1'}
     \| (I - P_{N} \De^{(n)} M) P_{N} q^{(n)}\| < \epsilon_{n} \text{ for }
     N \text{ sufficiently large.}
\end{equation}
As we are assuming that $q^{(n)}$ has support in $[n_{0}, \infty)$,
from the shift-invariance of $M$ and the observation made above that
$\De^{(n)}$ preserves signals with support in $[n_{0}, \infty)$, we see
that \eqref{est1'} can be rewritten as
\begin{equation}   \label{est1''}
    \| (I - P_{[n_{0}, N)} \De^{(n)} M) P_{[n_{0}, N)}q^{(n)} \| < \epsilon
\end{equation}
for $N$ sufficiently large.
  Note that $\operatorname{supp}  q^{(n)} \subset [n_{0}, \infty)$
  implies that $\operatorname{supp}
M q^{(n)} \subset [n_{0}, \infty)$ since $M$ by assumption is shift
invariant. We next use the identity
$$
  (I - P_{N})M P_{N} q^{(n)} = (I - P_{N}) M q^{(n)} - (I - P_{N}) M
  (I - P_{N}) q^{(n)}
$$
to get the estimate
\begin{align*}
    \| (I - P_{N}) M P_{N} q^{(n)} \| & \le
 \| (I - P_{N}) M q^{(n)} \| + \| (I-P_{N}) M (I -  P_{N}) q^{(n)} \| \\
 & \le \| (I - P_{N} ) M q^{(n)} \| + \| M \|  \| (I - P_{N}) q^{(n)} \|.
\end{align*}
As another consequence of the strong convergence of $P_{N}$ to
the identity operator, we see that, by choosing
$N$ still larger if necessary, we may arrange that in addition to
\eqref{est1''} we have
\begin{equation}  \label{est2}
    \| (I - P_{[n_{0}, N)}) M P_{[n_{0}, N)} q^{(n)} \| < \epsilon_{n}.
\end{equation}
Moreover, if we note that
\begin{align*}
  &  \| (I - P_{[n_{0}, N)} \De^{(n)} P_{[n_{0}, N)} M) P_{n_{0}, N)}
    q^{(n)} \|  \\
    & \quad \le
 \| (I - P_{[n_{0},N)} \De^{(n)} M) P_{[n_{0},N)} q^{(n)} \|
  + \| P_{[n_{0}, N)} \De^{(n)} (I - P_{[n_{0},N)}) M P_{[n_{0}, N)}
 q^{(n)} \|   \\
 & \quad \le  \| (I - P_{[n_{0},N)} \De^{(n)} M) P_{[n_{0},N)} q^{(n)} \|  +
 \| (I - P_{[n_{0},N)}) M P_{[n_{0}, N)} q^{(n)}  \|,
 \end{align*}
 we see as a consequence of the estimates \eqref{est1''} and
 \eqref{est2} that
 \begin{equation} \label{est3}
     \| (I - P_{[n_{0}, N)} \De^{(n)} P_{[n_{0}, N)} M) P_{[n_{0}, N)}
     q^{(n)}\| < 2 \epsilon_{n}.
 \end{equation}
Furthermore,  by rescaling and taking $N$ still larger if necessary,
we may assume in addition that $P_{[n_{0}, N)} q^{(n)}$ is a unit
vector.  By now setting $\widehat{q}^{(n)}=P_{[n_{0}, N)}
q^{(n)}$ and $\widehat \De^{(n)} = P_{[n_{0}, N)} \De P_{[n_{0},
N)}$, and rewriting \eqref{est3} and \eqref{est2} in the new notation, we
arrive at the following result of all this discussion:  {\em  for each $n_{0} \in
{\mathbb Z}_{+}$, there is a choice of sufficiently large $N \in
{\mathbb Z}_{+}$ so that the following holds true:
 there is a unit vector $\widehat{q}^{(n)}$ in $\ell^{2}[n_{0}, N)\otimes \cH^{\circ}_{\bS}$
and an operator $\widehat \De^{(n)}$ in
$\overline{\cB}\bDelta_{{\bG^\circ}} \cap \cL(\ell^{2}[n_{0},
N) \otimes \cH^{\circ}_{\bR}, \ell^{2}[n_{0},
N) \otimes \cH^{\circ}_{\bS})$  such that}
\begin{align}
    & \| (I - \widehat \De^{(n)} M) \widehat{q}^{(n)} \| < 2 \epsilon_{n}\quad\mbox{and}\quad
    \| (I - P_{[n_{0}, N)} ) M \widehat{q}^{(n)} \| < \epsilon_{n}.
    \label{est4a}
\end{align}

    By proceeding inductively, we may assume furthermore that the support
of $\widehat{q}^{(n)}$ is in an interval of the form $[t_{n}, t_{n+1}) \subset
{\mathbb Z}_{+}$ with $t_{0} = 0$ in such a way that these intervals
form a complete partition of ${\mathbb Z}_{+}$.  In this new notation
$\widehat{\De}^{(n)}$ is in
$\overline{\cB}\bDelta_{{\bG^\circ}} \cap \cL(\ell^{2}[t_{n}, t_{n+1}) \otimes \cH^{\circ}_{\bR}, \ell^{2}[t_{n}, t_{n+1}) \otimes \cH^{\circ}_{\bS})$.  If we
set $\widehat \De = \sum_{n=0}^{\infty} \widehat \De^{(n)} P_{[t_{n},
t_{n+1})}$, then $\| \widehat \De \| \le 1$ since each $\widehat
\De^{(n)}$ is contractive and furthermore $\widehat \De$ still has the
block diagonal structure to qualify as an element of
$\bDelta_{{\bG^\circ}}$, i.e., $\widehat \De \in
\overline{\cB}\bDelta_{{\bG^\circ}}$.  We now apply $(I - \widehat
\De M)$ to $\widehat{q}^{(n)}$ and estimate the norm of the result:
\begin{align*}
    \| ( I - \widehat \De M) \widehat{q}^{(n)} \| &  =
    \| (I - \widehat \De \{ P_{[t_{n}, t_{n+1})} + (I - P_{[t_{n},
    t_{n+1})})\} ) M ) \widehat{q}^{(n)} \|  \\
    & = \| (I - \widehat \De^{(n)} M) \widehat{q}^{(n)} - \widehat \De (I -
    P_{[t_{n}, t_{n+1})}) M \widehat{q}^{(n)} \|  \\
    & \le \| (I - \widehat \De^{(n)} M) \widehat{q}^{(n)} \| + \| (I - P_{[t_{n},
    t_{n+1})}) M \widehat{q}^{(n)} \| \\
    & < 2 \epsilon_{n} + \epsilon_{n} = 3 \epsilon_{n}
\end{align*}
where we used \eqref{est4a} for the last inequality.
As each $\widehat{q}^{(n)}$ is a unit vector and $3\epsilon_{n} \to 0$ as $n
\to \infty$, we conclude that $I - \widehat \De M$ cannot be
invertible, despite the fact that $\widehat \De \in \overline{\cB}
\bDelta_{\overline{\bG}}$.  This contradiction to our underlying
hypothesis completes the proof of Lemma \ref{L:DP}.
\end{proof}

\begin{remark} \label{R:mu-to-ncBRL}
   {\em  We have seen that the strict Bounded Real Lemma
    (Theorem \ref{T:strictBRL}) with the help of Lemma \ref{L:DP}
    implies the Main Result (Theorem \ref{T:I}).  It is of
    interest that conversely Theorem \ref{T:I} implies Theorem
    \ref{T:strictBRL} by a simple direct argument as follows.
    Suppose that we are given a colligation matrix $\left[
    \begin{smallmatrix} A & B \\ C & D \end{smallmatrix} \right]$ as
	in Theorem \ref{T:strictBRL}.  By hypothesis we have
$$
\| I_{\ell^{2}} \otimes D +
(I_{\ell^{2}} \otimes C) (I - \De_{1} (I_{\ell^{2}} \otimes A))^{-1} \De_{1}
(I_{\ell^{2}} \otimes B) \| \le \rho < 1
$$
for all $\De_{1} \in \overline{\cB} \bDelta_{{\bG^\circ}}$.  By a
Schur-complement argument (see  \cite[Theorem 11.7]{ZDG} known as the Main Loop
Theorem), this is the same as
the block $2 \times 2$ matrix
$\left[\begin{smallmatrix} I & 0 \\ 0 & I \end{smallmatrix} \right]
-\left[ \begin{smallmatrix} \De_{1}
    & 0 \\ 0 & \De_{2} \end{smallmatrix} \right]
    \left[ \begin{smallmatrix} A
    & B \\ C & D \end{smallmatrix} \right]$ being invertible for
    all $\De_{1} \in \overline{\cB} \bDelta_{{\bG^\circ}}$ and
    $\De_{2} \in \overline{\cB} \bDelta_{\rm full}$, where we set
    $\bDelta_{\rm full}$ equal to the set of all operators from $\cU$
    to $\cY$.  This in turn is the same as the statement
 $$
   \mu_{\bDelta_{{\bG}^\circ} \oplus \bDelta_{\rm
   full}} \left(I_{\ell^{2}} \otimes \left[ \begin{smallmatrix} A & B \\ C
   & D \end{smallmatrix} \right] \right) < 1.
 $$
An application of Theorem \ref{T:I} now tells us that there exists a
positive-definite matrix $\Gamma^{\circ}_{p}$ on $\cH^{\circ}_{p}$
for each $p \in \bP$  and a positive real number $r > 0$ so that
$$
\begin{bmatrix} A^{*} & C^{*} \\ B^{*} & D^{*} \end{bmatrix}
    \begin{bmatrix} \bigoplus_{r \in \bR} \Gamma^{\circ}_{[r]} & 0 \\
    0 & s I_{\cY} \end{bmatrix}
  \begin{bmatrix} A & B \\ C & D \end{bmatrix} -
      \begin{bmatrix} \bigoplus_{s \in \bS} \Gamma^{\circ}_{[s]} & 0
	  \\ 0 & s I_{\cU} \end{bmatrix} \prec 0.
$$
If we divide out by the positive numbers $s$ and replace
$\Gamma^{\circ}_{p}$ by $\frac{1}{s} \cdot \Gamma^{\circ}_{p}$ for
each $p \in \bP$, we arrive  at exactly statement (3) in Theorem
\ref{T:strictBRL}.

This analysis can be taken one step further to get a new proof of the
strict version of the realization result Theorem \ref{T:BGM2} as
follows.  Given a rational formal power series in the strict
Schur-Agler class, using results from
\cite{BGM1} (closely related to the much earlier realization results
of Fliess \cite{Fliess}),  one can obtain a finite-dimensional
colligation matrix $\bU$ as in \eqref{col} giving rise to a
realization \eqref{real} for $S(z)$.  Then use the strict Bounded
Real Lemma (which as we have just seen is a direct consequence of the
$\widetilde \mu = \widehat \mu$ result Theorem \ref{T:I}) to obtain a
structured state-space similarity transforming the colligation matrix
$\bU = \left[ \begin{smallmatrix} A & B \\ C & D \end{smallmatrix}
\right]$ to the strictly contractive colligation matrix $\bU' =
\left[ \begin{smallmatrix} A' & B' \\ C' & D' \end{smallmatrix}
\right]$. Then $\bU'$ is a strictly contractive colligation matrix
with transfer function \eqref{real} (with $\bU'$ in place of $\bU$)
equal to $S(z)$, and the strict version of Theorem \ref{T:BGM2}
follows.
}\end{remark}

\begin{remark}  \label{R:infdim}
{\em It is possible to note now that Theorem \ref{T:I} cannot be true if
any of the partial state spaces $\cH_{p}$ is allowed to be
infinite-dimensional and/or if the graph $\bG$ is allowed to be
infinite.  Indeed it is known (see \cite{AKP}) that the Bounded Real
Lemma fails if the state space is allowed to be infinite-dimensional;
the proof relies on the State Space Similarity Theorem which in turn
only guarantees a possibly unbounded pseudo-similarity in the
infinite-dimensional setting rather than a properly bounded and
boundedly invertible similarity.  A simple adaptation of the example
given in \cite{AKP} shows that the strict
Bounded Real Lemma also fails in the case of of infinite-dimensional
state space as well. By the preceding Remark \ref{R:mu-to-ncBRL}, Theorem
\ref{T:I}  is equivalent to the Bounded Real Lemma in the free
noncommutative setting.  We conclude that Theorem \ref{T:I} cannot
hold in general when $\cH_{p}$ is allowed to be infinite-dimensional
or if the graph $\bG$ is allowed to have infinitely many connected
components.

It is interesting to note however that Lemma \ref{L:DP} apparently
does not require the finite-dimensionality of the coefficient spaces
$\cH^{\circ}_{\bR}$ and $\cH^{\circ}_{\bS}$; one only requires that
the operator $M$ be shift-invariant with respect to the pair of
shifts $(V_{\bS}, V_{\bR})$, even possibly of infinite multiplicity.

In our second proof of Theorem \ref{T:I} in Section \ref{S:DPproof},
the reader will notice several places where the finite-dimensionality
of the coefficient spaces $\cH^{\circ}_{p}$ and the finiteness of the
graph are used---see in particular the assumed equivalence of
Hilbert-Schmidt norm and operator norm in the verification of Step 1
and in the estimate \eqref{estimate}.
} \end{remark}

\begin{remark}  \label{R:graphrole} {\em In the {\em graded version} of the
 structured ball
$$
\cB \bDelta_{{\bG^\circ}} =
    \{ L_{\bG^{\circ}}(Z) = \sum_{e \in \bE} Z_{e} \otimes
    L_{\bG^{\circ},e}\colon Z_{e} \in\cL(\cK),\,
    \|L_{\bG^{\circ}}(Z)\| < 1\},
$$
one restricts $Z_{e}$ to finite square matrices $Z_{e}
\in {\mathbb C}^{n \times n}$ for every matrix size $n=1,2,\dots$
rather than  letting $Z_{e}$ range over all bounded linear operators
on a fixed infinite-dimensional separable Hilbert space $\cK$.  The
preimage of this graded structured ball under the pencil, namely
$$
\cB \bDelta^{\rm pre, graded}_{{\bG^\circ}}
= \{ Z = (Z_{e})_{e \in \bE} \colon Z_{e} \in {\mathbb C}^{n \times
n} \text{ for } n = 1,2,\dots,\, \|L_{\bG^{\circ}}(Z) \| < 1\}
$$
corresponds to the noncommutative pencil ball studied by Helton,
Klep, McCullough and Slinglend in \cite{HKMcCS2009, HKMcC2011}.
Actually these authors consider the more general setting where the
formal pencil $L_{\bG^{\circ}}(z)$ is replaced by a general formal
pencil  $L(z) = \sum_{e \in \bE} L_{e} z_{e}$ where here $\bE$ is now just
a convenient index set and the coefficients $L_{e}$ no longer have
any connection with an underlying graph.  More generally, Agler and
McCarthy \cite{AMcCGlobal} obtained a graded version of the
realization result Theorem \ref{T:BGM2}, where the structured matrix
pencil $L_{\bG^{\circ}}(z)$ is replaced by an arbitrary
formal polynomial $\delta(z)$ with matrix coefficients, thereby
obtaining a graded noncommutative analogue of the commutative result
of Ball-Bolotnikov \cite{BB} and Ambrozie-Timotin \cite{AT}.
This more general formalism led to new results on polynomial approximation and rigidity results for proper analytic maps between such domains, respectively for the noncommutative setting.
We point out here, however, that when one replaces the
structure noncommutative  pencil $L_{\bG^{\circ}}(z)$ by a general
noncommutative pencil $L(z)$ or a general matrix noncommutative
polynomial $\delta(z)$,  one loses other results involving the more
detailed structure of the associated noncommutative linear systems,
specifically, the State Space Similarity Theorem from \cite{BGM1} and
hence also the Bounded Real Lemma from \cite{BGM3} (Theorem
\ref{T:strictBRL}).
}\end{remark}

\section{Noncommutative structured singular value versus diagonal scaling: a direct
convexity argument for the higher multiplicity case}  \label{S:DPproof}

In this section we present our second proof of the Main Result
(Theorem \ref{T:I}), this time based on the convexity-analysis
approach of Dullerud and Paganini \cite{Paganini, DP}.
In fact the approach enables one to prove the following more general
formulation of Theorem \ref{T:I}. Note that Theorem \ref{T:I} follows from the
following Theorem \ref{T:I'} by setting $M = I_{\ell^{2}} \otimes
M^{\circ}$.

\begin{theorem}   \label{T:I'}
    Let $\overline{\bG}$, $\bG$, and $\bG^\circ$ be as in Section
    \ref{S:enhanced} with $\cK = \ell^{2}$, let $M$ be a linear
    operator from the space
 $$
   \cH_{\bS} = \ell^{2} \otimes \cH^{\circ}_{\bS} = \bigoplus_{p \in
   \bP}\left( \ell^{2}  \otimes \widetilde \cH_{\bS,p} \otimes
   \cH^{\circ}_{p} \right)
 $$
 to the space
$$
   \cH_{\bR} = \ell^{2} \otimes \cH^{\circ}_{\bR} = \bigoplus_{p \in
   \bP} \left( \ell^{2}  \otimes \widetilde \cH_{\bR,p} \otimes
   \cH^{\circ}_{p} \right)
 $$
 which is shift-invariant:
 $$V_{\bR} M = M V_{\bS}\quad \text{where}\quad V_{\bR} = V \otimes
 I_{\cH^{\circ}_{\bS}},\quad  V_{\bS} = V \otimes I_{\cH^{\circ}_{\bR}}
 $$
 with $V$ is the unilateral shift operator on $\ell^{2}$.  Assume
 also that
 \begin{enumerate}
     \item[(i)] the graph $\bG^{\circ}$ has only finitely many
     components, and
     \item[(ii)] each coefficient space $\cH^{\circ}_{p}$ is
     finite-dimensional.
  \end{enumerate}
 Then
 $$
 \mu_{\bDelta_{\overline{\bG}}}(M) =
 \widehat \mu_{\bDelta_{\overline{\bG}}}(M).
 $$
 In particular, the following conditions are equivalent:
\begin{enumerate}
     \item  $\mu_{\bDelta_{\overline{\bG}}}(M) < 1$, i.e.,
     $(I - \De M)^{-1}$ exists for all $\De = {\rm diag}_{p \in
     \bP} W_{p} \otimes I_{\cH^{\circ}_{p}}$ with $W_{p} \in
     \rho \overline{\cB}\cL(\ell^{2} \otimes \widetilde \cH_{\bR,p},
     \ell^{2} \otimes \widetilde \cH_{{\bS,p}})$ for some $\rho > 1$.

     \item $\widehat \mu_{\bDelta_{\overline{\bG}}}(M) < 1$, i.e.,
     there exists operators $\Gamma^{\circ}_{p} \succ 0$ on
     $\cH^{\circ}_{p}$ so that
\begin{equation}   \label{muhatcond}
  M^{*} \left( \bigoplus_{p \in \bP} I_{\ell^{2} \otimes \widetilde
  \cH_{\bR,p}} \otimes \Gamma^{\circ}_{p} \right) M -
  \left( \bigoplus_{p \in \bP} I_{\ell^{2} \otimes \widetilde
  \cH_{\bS,p}} \otimes \Gamma^{\circ}_{p} \right) \prec 0.
\end{equation}
\end{enumerate}
\end{theorem}

\begin{proof}
Following the argumentation in Remark \ref{R:reduct}, a scaling argument gives that the equality
    $\mu_{\bDelta_{\overline{\bG}}}(M) = \widehat
    \mu_{\bDelta_{\overline{\bG}}}(M)$ is equivalent to:
$\mu_{\bDelta_{\overline{\bG}}}(M) < 1$ $\Leftrightarrow$ $\widehat
\mu_{\bDelta_{\overline{\bG}}}(M) < 1$.  This latter statement in
turn is equivalent to the equivalence of the two statements (1) and
(2) in the statement of the theorem. Thus it suffices to show the
equivalence of (1) and (2).  If (2) holds, then $M$ is
$\overline{\bG}$-structured-similar to a strict contraction $M'$ from
which (1) follows.  We conclude that it suffices to show that (1)
$\Rightarrow$ (2).

Toward this goal, we assume that we are given $M$ for which (1)
holds.  Let us use the short-hand notation
\begin{equation}   \label{shorthand}
U_{\bR,p} = U_{\ell^{2}\otimes \widetilde \cH_{\bR,p}, \cH^{\circ}_{p}}, \quad
U_{\bS,p} = U_{\ell^{2}\otimes \widetilde \cH_{\bS,p}, \cH^{\circ}_{p}}
\end{equation}
for the identification maps between tensor product spaces and
Hilbert-Schmidt operators given in Proposition \ref{P:tensor}.
We introduce maps $\phi_{p} \colon \cH_{\bS} \to
\cC_{1}(\cH^{\circ}_{p})$ by
\begin{equation} \label{phip}
\phi_{p} \colon  h \mapsto
\left(U_{\bR,p}[P_{\cH_{\bR,p}} Mh] \right)^{*}
U_{\bR,p}[P_{\cH_{\bR,p}} Mh ]
 -\left( U_{\bS,p}[P_{\cH_{\bS,p}} h ]\right)^{*}
U_{\bS,p}[P_{\cH_{\bS,p}} h ]
\end{equation}
In addition introduce sets of operator tuples
\begin{align}
\nabla &=\{(\phi_p(h))_{p \in \bP}\,\colon\, h\in\cH_S,\,
\|h\|=1\},\label{nabla}\\
\Pi &=\{(L_p)_{p \in \bP}\, \colon\, L_p\in \cC_{1}(\cH_p^\circ),\, L_p
\succeq 0,\,   p\in \bP \}.\label{Pi}
\end{align}

The connection between the quadratic forms $\phi_{p}$ and the condition
$\widehat \mu(M) < 1$ (condition (1) in Theorem \ref{T:I'}) is as follows.

\begin{lemma}  \label{L:1}
    Assume that $\mu_{\bDelta_{\overline{\bG}}}(M) < 1$ (i.e., condition (1) in Theorem
    \ref{T:I'} is satisfied).  Then
 \begin{equation}  \label{NablaPi}
     \nabla \cap \Pi = \emptyset,
 \end{equation}
    i.e.,  there cannot exist a nonzero $h \in \cH_{\bS}$ such that
    $\phi_{p}(h) \succeq 0$ for each $p \in \bP$.
\end{lemma}

\begin{proof}  First note that each $\phi_{p}$ is homogeneous of
    degree 2: $\phi_{p}(\alpha h) = |\alpha|^{2} \phi_{p}(h)$ for
    $\alpha \in {\mathbb C}$.  Thus the existence of a nonzero $h \in
    \cH_{\bS}$ with $\phi_{p}(h) \succeq 0$ for all $p$ implies that the
    normalization $\widetilde h = \|h\|^{-1}h$ of $h$ is a unit vector which
    satisfies $\phi_{p} (\widetilde h) \succeq  0$ for all $p$.
    Thus the existence of a nonzero $h \in \cH_{\bS}$ with
    $\phi_{p}(h) \succeq 0$ for all $p$ is equivalent to $\nabla \cap
    \Pi$ being nonempty.

    To prove the lemma we proceed by contradiction.
     Suppose that there is a
    nonzero $h \in \cH_{\bS}$ such that $\phi_{p}(h) \succeq 0$ for
    all $p$.  By Proposition \ref{P:Douglas} we can find a contraction
    $W_{p} \in \cL(\ell^{2} \otimes \widetilde
    \cH_{\bR, p}, \ell^{2} \otimes \widetilde \cH_{\bS,p})$ so that
$$
  W_{p} \otimes I_{\cH^{\circ}_{p}} \colon P_{\cH_{\bR,p}} M h \to
  P_{\cH_{\bS,p}} h.
$$
Then $\De = \bigoplus_{p \in \bP} \left( W_{p} \otimes
I_{\cH^{\circ}_{p}} \right)$ is in $\overline{\cB}
\bDelta_{\overline{\bG}}$ and $h$ is in the kernel of $(I - \De M)$.
It follows that $I - \De M$ is not invertible, i.e., condition (1) in
Theorem \ref{T:I'} is violated.
\end{proof}

The connection of the quadratic forms $\phi_{p}$ with the condition
$\widehat \mu_{\bDelta_{\bG}}(M) < 1$ (condition (2)
in Theorem \ref{T:I'}) is as follows.

\begin{lemma}  \label{L:2}
   The condition
    $\widehat \mu_{\bDelta_{\bG}}(M) < 1$ holds, i.e.,  for each $p \in
    \bP$ there exists $\Gamma_{p}^{\circ} \succ 0$ on
    $\cH^{\circ}_{p}$ so that \eqref{muhatcond} holds, if and only if
    either of the following two equivalent conditions holds:
    \begin{enumerate}
    \item
    There exists $\epsilon > 0$ and strictly positive definite
    operators $\Gamma^{\circ}_{p} \succ 0$ on $\cH^{\circ}_{p}$ for
    each $p \in \bP$ so that
    \begin{equation}   \label{L2formula}  \sum_{p \in \bP}  {\rm tr} \left(
    \Gamma^{\circ}_{p} \, \phi_{p}(h) \right) \le - \epsilon
    \|h\|^{2}\qquad (h \in \cH_{\bS}).
    \end{equation}

    \item The sets $\nabla$ and $\Pi$ are strictly separated in the
    following sense:  there exists operators
    $\Gamma^{\circ}_{p}$ on $\cH^{\circ}_{p}$ for $p \in \bP$
    and real numbers $\alpha < \beta$ so that
    \begin{equation}  \label{strictsep}
\sum_{p \in \bP} {\rm Re}\, {\rm tr}(\Gamma^{\circ}_{p} \, K_{p}) \le \alpha <
\beta \le \sum_{p \in \bP} {\rm Re}\,{\rm tr}(\Gamma^{\circ}_{p}\, L_{p})
\ ((K_{p})_{p \in \bP} \in \nabla ,\, (L_{p})_{p \in \bP}
    \in \Pi).
    \end{equation}
 Furthermore, whenever this is the case, it can be
    arranged that $\beta = 0$ and $\Gamma^{\circ}_{p} \succ 0$ and
    then \eqref{strictsep} can be written without the real-part
    qualifier:
     \begin{equation}  \label{strictsep'}
\sum_{p \in \bP}  {\rm tr}(\Gamma^{\circ}_{p} \, K_{p}) \le \alpha <
0 \le \sum_{p \in \bP} {\rm tr}(\Gamma^{\circ}_{p}\, L_{p} )
\quad
((K_{p})_{p \in \bP} \in \nabla,\, (L_{p})_{p \in \bP}
    \in \Pi).
    \end{equation}
 \end{enumerate}
\end{lemma}

\begin{proof}
    Rewrite \eqref{muhatcond} as a quadratic form condition:
\begin{equation}   \label{rewrite}
\left\langle \left[ M^{*} \left( \bigoplus_{p \in \bP} I_{\ell^{2}
\otimes \widetilde \cH_{\bR,p}} \otimes \Gamma^{\circ}_{p} \right) M -
\left( \bigoplus_{p \in \bP} I_{\ell^{2} \otimes \widetilde
\cH_{\bS,p}} \otimes \Gamma_{p}^{\circ} \right) \right] h, h
\right\rangle \le -\epsilon \|h\|^{2}.
\end{equation}
The left-hand side of this inequality can be rewritten as a
difference of sums:
\begin{align*}
    & \left\langle \left[ M^{*} \left( \bigoplus_{p \in \bP} I_{\ell^{2}
\otimes \widetilde \cH_{\bR,p}} \otimes \Gamma^{\circ}_{p} \right) M -
\left( \bigoplus_{p \in \bP} I_{\ell^{2} \otimes \widetilde
\cH_{\bS,p}} \otimes \Gamma_{p}^{\circ} \right) \right] h, h
\right\rangle \\
& \quad = \sum_{p \in \bP} \left\langle \left( I_{\ell^{2} \otimes
\widetilde \cH_{\bR,p}} \otimes \Gamma^{\circ}_{p} \right)
P_{\cH_{\bR,p}} M h, P_{\cH_{\bR,p}}M h \right\rangle_{\cH_{\bR,p}} \\
& \quad \quad \quad \quad
- \sum_{p \in \bP} \left\langle \left( I_{\ell^{2} \otimes \widetilde
\cH_{\bS,p}} \otimes \Gamma^{\circ}_{p} \right) P_{\cH_{\bS,p}}h,
P_{\cH_{\bS,p}} h \right\rangle_{\cH_{\bS,p}}
\end{align*}
Now note that
\begin{align*}
  & \left\langle \left( I_{\ell^{2} \otimes \widetilde \cH_{\bR,p}}
    \otimes \Gamma^{\circ}_{p} \right) P_{\cH_{\bR,p}} M h,
    P_{\cH_{\bR,p}} M h \right\rangle  \\
 & \quad =  \left\langle U_{\ell^{2} \otimes \cH_{\bR, p}, \cH^{\circ}_{p}}\left[
 \left(I_{\ell^{2} \otimes \widetilde \cH_{\bR,p}} \otimes
 \Gamma^{\circ}_{p} \right) P_{\cH_{\bR,p}} M h \right], U_{\ell^{2}
 \otimes \widetilde \cH_{\bR,p}, \cH^{\circ}_{p}}[ P_{\cH_{\bR,p}} M
 h]  \right\rangle  \\
 & \quad = \left\langle U_{\ell^{2} \otimes \cH_{\bR, p}, \cH^{\circ}_{p}}\left[
 P_{\cH_{\bR,p}} M h \right]
 (\Gamma^{\circ}_{p})^T, U_{\ell^{2}
 \otimes \widetilde \cH_{\bR,p}, \cH^{\circ}_{p}}[ P_{\cH_{\bR,p}} M
 h ] \right\rangle  \text{ (by property \eqref{intertwine})}  \\
 & \quad =  {\rm tr} \left( (\Gamma^{\circ}_{p})^T  \,
 U_{\ell^{2} \otimes \cH_{\bR, p}, \cH^{\circ}_{p}}\left[
 P_{\cH_{\bR,p}} M h \right]^{*} U_{\ell^{2} \otimes \cH_{\bR, p}, \cH^{\circ}_{p}}\left[
 P_{\cH_{\bR,p}} M h \right] \right)
 \end{align*}
 A similar calculation gives that
 \begin{align*}
  &   \left\langle \left( I_{\ell^{2} \otimes \widetilde
\cH_{\bS,p}} \otimes \Gamma^{\circ}_{p} \right) P_{\cH_{\bS,p}}h,
P_{\cH_{\bS,p}} h \right\rangle_{\cH_{\bS,p}} \\
& \qquad\qquad =
{\rm tr} \left( (\Gamma^{\circ}_{p})^T\,
U_{\ell^{2} \otimes \cH_{\bS,p}, \cH^{\circ}_{p}}[ P_{\cH_{\bS,p}} h]^{*}
U_{\ell^{2} \otimes \cH_{\bS,p}, \cH^{\circ}_{p}}[ P_{\cH_{\bS,p}} h]
\right)
\end{align*}
Putting the pieces together, we see that the condition
\eqref{rewrite} collapses to \eqref{L2formula} (with
$(\Gamma^{\circ}_{p})^T$ in place of $\Gamma^{\circ}_{p}$).
Since the conjugation operator preserves strict positive-definiteness
and is involutive, having $(\Gamma^{\circ}_{p})^T$ in the
formula rather than $\Gamma^{\circ}_{p}$ does not affect the result.
Conversely, by reversing the steps in the argument, one can derive
\eqref{muhatcond} from \eqref{L2formula}.  This completes the proof
of the equivalence of \eqref{muhatcond} and \eqref{L2formula}.

It remains to argue the equivalence of conditions (1) and (2) in
Lemma \ref{L:2}. Assume that condition (1) holds, i.e., that there are positive definite
operators
$\Gamma^{\circ}_{p}$ on $\cH^{\circ}_{p}$ for each $p \in \bP$ so
that \eqref{L2formula} holds.  Each $(K_{p})_{p \in \bP}$ in $\nabla$
has the form $K_{p} = \phi_{p}(h)$ for an $h \in \cH_{\bS}$ with $\|
h \|_{\cH_{\bS}} = 1$.  Using this connection between $(K_{p})_{p \in
\bP}$ in $\nabla$ and $h$ in formula \eqref{L2formula} gives
$$
\sum_{p \in \bP} {\rm tr}(\Gamma^{\circ}_{p} \, K_{p}) =
\sum_{p \in \bP} {\rm tr}( \Gamma^{\circ}_{p} \, \phi_{p}(h)) \le
-\epsilon =: \alpha < 0.
$$
Furthermore, for $\Gamma^{\circ}_{p} \succ 0$ and $L_{p} \succeq 0$,
it is automatic that ${\rm tr}(\Gamma^{\circ}_{p} \, L_{p}) \ge 0$
for each $p$, and hence \eqref{strictsep} follows with $\be=0$ (and all $\Ga_p^\circ \succ0$).

Conversely, suppose that there are operators
$\Gamma^{\circ}_{p}$ on $\cH^{\circ}_{p}$ and numbers $\alpha < \beta$ so that
\eqref{strictsep} holds. As in general ${\rm tr} (X^{*}) =
\overline{{\rm tr}(X)}$ and all components of elements of $\nabla$ and of $\Pi$
are selfadjoint, we see that $((\Gamma^{\circ}_{p})^{*})_{p \in \bP}$
satisfies \eqref{strictsep} whenever $(\Gamma^{\circ}_{p})_{p \in
\bP}$ does.  By the convexity of the conditions in \eqref{strictsep},
we may replace each $\Gamma^{\circ}_{p}$ by ${\rm Re}
\Gamma^{\circ}_{p} = \frac{1}{2}(\Gamma^{\circ}_{p} +
(\Gamma^{\circ}_{p})^{*})$ and still have a solution of
\eqref{strictsep}.  Once this is done, then the presence of the
real-part symbol in the formula is redundant and may be removed.
At this stage we know:  each $\Gamma^{\circ}_{p}$ is selfadjoint and
\begin{equation}   \label{strictsep''}
\sum_{p \in \bP}  {\rm tr}(\Gamma^{\circ}_{p} \, K_{p}) \le \alpha <
\beta \le \sum_{p \in \bP} {\rm tr}(\Gamma^{\circ}_{p}\, L_{p} )
\text{ for } (K_{p})_{p \in \bP} \in \nabla \text{ and } (L_{p})_{p \in \bP}
    \in \Pi.
    \end{equation}

A particular consequence of \eqref{strictsep''} is that
$$
\beta \le \sum_{p \in \bP} {\rm tr}(\Gamma^{\circ}_{p}\, L_{p} )
\quad\text{ for }\quad L_{p} \succeq 0.
$$
Fix a $p_{o} \in P$ and apply this condition to the particular case
where $L_{p}= 0$ for all $p \ne p_{0}$.  Then we see that ${\rm
tr}(\Gamma^{\circ}_{p_{0}}\, L_{p_{0}}) \ge 0$ for all $L_{p_{0}} \succeq
0$ on $\cH^{\circ}_{p_{0}}$.  Apply this condition to the particular
case where $L_{p_{0}} = v v^{*}$ for a unit vector $v \in
\cH^{\circ}_{p}$.  If it were not the case that $\langle
\Gamma^{\circ}_{p_{0}}v, v \rangle = {\rm
tr}(\Gamma^{\circ}_{p_{0}}\, L_{p_{0}}) \ge 0$, then we could rescale
$v$ to make ${\rm tr}(\Gamma^{\circ}_{p_{0}}\, L_{p_{0}})$ tend
as close as we like to $-\infty$, in particular, to achieve a value
strictly less than $\beta$ in violation of condition
\eqref{strictsep''}.  We conclude that $\Gamma^{\circ}_{p_{0}} \succeq 0$
for each $p_{0} \in \bP$ and that there is no loss of generality in
taking $\beta = 0$ and then $\alpha < 0$.

It remains to see that $\Gamma^{\circ}_{p} \succ 0$ for each $p \in \bP$.
Toward this end, note that $\nabla$ is a bounded subset of
$(\cC_{1}(\cH^{\circ}_{p}))_{p \in \bP}$ and $\alpha < 0$.  Hence we
may perturb each $\Gamma^{\circ}_{p}$ to $\Gamma^{\circ}_{p} + \delta
I_{\cH^{\circ}_{p}}$ for some number $\delta > 0$ sufficiently small
while maintaining ${\rm tr} ( \Gamma^{\circ}_{p} \, K_{p} ) \le
\alpha' = \alpha/2 < 0$.  With these adjustments, we arrive at the
existence of adjusted $\Gamma^{\circ}_{p} \succ 0$ and adjusted
$\alpha < 0$ so that \eqref{strictsep'} holds.

Once we have the validity of \eqref{strictsep'}, it is a simple
matter to make the substitution $K_{p} = \phi_{p}(h)$ with $h$ equal
to a unit vector in $\cH_{\bS}$ to arrive at \eqref{L2formula} for
the case where $h$ is a unit vector.  As both sides of
\eqref{L2formula} are quadratic in rescalings of $h$, the general
case of \eqref{L2formula} now follows as well.
\end{proof}

The results of Lemma \ref{L:1} and \ref{L:2} show the apparent
gap between the conditions $\mu_{\bDelta_{\overline{\bG}}}(M) < 1$
($\nabla \cap \Pi = \emptyset$) and $\widehat
\mu_{\bDelta_{\overline{\bG}}}(M) < 1$ ($\nabla$ and $\Pi$ strictly
separated).  In fact, the strict separation condition
\eqref{strictsep} says much more, namely:
\begin{equation}
    \overline{\rm co}\, \nabla \cap \Pi = \emptyset
\end{equation}
where $\overline{\rm co} \, \nabla$ is the closed convex hull of the set
$\nabla$.  This suggests some elementary convexity analysis.  We view
$\cX: = (\cC_{1}(\cH^{\circ}_{p}))_{p \in \bP}$ as a linear topological vector
space (a bit of an overblown statement since we are assuming that it
is finite dimensional) with dual space viewed as operator tuples
$\cX^{*} : =(\cL(\cH^{\circ}_{p}))_{p \in \bP}$ with duality pairing given via the
trace:
$$ \left\langle (\Gamma^{\circ}_{p})_{p \in \bP}, (T_{p})_{p \in \bP}
\right\rangle = \sum_{p \in \bP} {\rm tr}(\Gamma^{\circ}_{p}\, T_{p} )
$$
where
$$
(\Gamma^{\circ}_{p})_{p \in \bP} \in (\cL(\cH^{\circ}_{p}))_{p \in
\bP}, \quad
(T_{p})_{p \in \bP} \in (\cC_{1}(\cH^{\circ}_{p}))_{p \in
\bP}.
$$
Note that $\Pi$ and $\overline{\rm co} \, \nabla$ are  closed convex sets in $\cX$
and furthermore, as $\overline{\rm co} \, \nabla$ is closed and bounded
in the finite-dimensional Banach space $\cX$, $\overline{\rm co} \,
\nabla$ is also compact.  We may therefore apply a Hahn-Banach
separation theorem (see Theorem 3.4 part (b) in \cite{Rudin}) to
conclude: {\em if $\overline{\rm co} \, \nabla \cap \Pi = \emptyset$, then
$\overline{\rm co} \, \nabla$ and $\Pi$ (and hence also $\nabla$ and $\Pi$)
are strictly separated.}  It then follows as a consequence of Lemma
\ref{L:2} that $\widehat \mu_{\bDelta_{\overline{\bG}}}(M) < 1$.
Hence to complete the proof of Theorem \ref{T:I'}, it remains only to
show:
\begin{equation}   \label{toshow}
   \mu_{\bDelta_{\overline{\bG}}}(M) < 1\quad \Rightarrow \quad \overline{\rm
    co} \, \nabla \cap \Pi = \emptyset.
\end{equation}
The verification of the implication \eqref{toshow} proceeds in two
steps:
\begin{enumerate}
    \item[\textbf{Step 1:}]  If $\mu_{\bDelta_{\overline{\bG}}}(M) <
    1$, then the necessary condition \eqref{NablaPi} holds in the
    stronger form
    \begin{equation}  \label{distNablaPi}
	\overline{\nabla} \cap \Pi = \emptyset.
\end{equation}

    \smallskip

    \item[\textbf{Step 2:}]  If $\mu_{\bDelta_{\overline{\bG}}}(M) <
    1$, then the closure $\overline{\nabla}$ of $\nabla$ is convex, i.e.,
    \begin{equation}  \label{closedconvex}
	\overline{\rm co} \, \nabla =  \overline{\nabla}.
\end{equation}
\end{enumerate}

\paragraph{\bf Verification of Step 1:}
The argument is modeled on the proof of Lemma B.1 in \cite{DP}
inspired in turn by the earlier work of Shamma \cite{Shamma}; the
reader will see that the argument also has some elements in common with the proof of Lemma
\ref{L:DP} above.

To streamline the proof, let us use the short-hand notation
\eqref{shorthand}.  Recall that $V$ denotes the shift operator on
$\ell^{2}$;
 let us introduce additional short-hand notation for the higher-multiplicity
 shift operators
 \begin{align*}
  & \widetilde V_{\bR}: = V \otimes I_{\widetilde \cH_{\bR,p}}
   \text{ on } \ell^{2} \otimes \widetilde \cH_{\bR,p}, \quad
    \widetilde V_{\bS}: = V \otimes I_{\widetilde \cH_{\bS,p}}
   \text{ on } \ell^{2} \otimes \widetilde \cH_{\bS,p},  \\
 &   V_{\bR}: = V \otimes I_{ \cH^{\circ}_{\bR,p}}
   \text{ on } \ell^{2} \otimes  \cH^{\circ}_{\bR,p}, \quad
     V_{\bS}: = V \otimes I_{ \cH^{\circ}_{\bS,p}}
   \text{ on } \ell^{2} \otimes  \cH^{\circ}_{\bS,p}
   \end{align*}
 where we recall that $\cH^{\circ}_{\bR,p} = \widetilde \cH_{\bR,p}
 \otimes \cH^{\circ}_{p}$ and $\cH^{\circ}_{\bS,p} = \widetilde
 \cH_{\bS,p} \otimes \cH^{\circ}_{p}$.  For brevity we use the same notation for
 the case where $\cH_{\bS,p}$ is replaced by $\cH_{\bS}$ or
 $\cH_{\bR,p}$ is replaced by $\cH_{\bR}$; the meaning will be clear
 from the context.
 Then we have the identities
 \begin{equation}   \label{shiftintertwine}
 \widetilde V_{\bR} U_{\bR,p}[h_{\bR}] = U_{\bR,p}[V_{\bR,p}
 h_{\bR}], \quad
 \widetilde V_{\bS} U_{\bS,p}[h_{\bS}] = U_{\bS,p}[V_{\bS}
 h_{\bS}]
 \end{equation}
 for $h_{\bR} \in \cH_{\bR,p}$ and $h_{\bS} \in \cH_{\bS, p}$ as a
 consequence of property \eqref{intertwine} in Proposition
 \ref{P:tensor2}.

Note that the condition \eqref{distNablaPi} can otherwise be formulated as
$\operatorname{dist} (\nabla \cap \Pi) > 0$ where the distance can be
measured via any convenient norm on the trace-class operator-tuples
$(\cC_{1}(\cH^{\circ}_{p}))_{p \in \bP}$; note that as part of our
assumptions is that $\dim \cH^{\circ}_{p} < \infty$ for each $p$, all
the norms on $\cH^{\circ}_{p}$ are equivalent. For convenience we work with
the operator norm.

We proceed by
contradiction.  Suppose that ${\rm dist}(\nabla, \Pi) = 0$.  The idea
is to construct a $\De \in \overline{\cB}\overline{\bDelta}_{\overline{\bG}}$ so that
$I - \De M$ is not (boundedly) invertible. Let $\{\epsilon_{n}\}_{n \in
{\mathbb Z}_{+}}$ be a sequence of positive real numbers with
$\epsilon_n \to 0$ as $n \to \infty$.  Since ${\rm dist}(\nabla, \Pi) =
0$, we can find a unit vector $q^{(n)} \in \cH_{\bS}$ and an
operator-tuple $(L^{(n)}_{p})_{p \in \bP} \in \Pi$ so that
$$
  \| \phi_{p}(q^{(n)}) - L^{(n)}_{p} \|_{\cL(\cH^{\circ}_{p})} < \epsilon_{n}^{2}
  \text{ for each } p \in \bP.
$$
Since $\phi_{p}(q^{(n)})$ and $L^{(n)}_{p}$ are all selfadjoint, this
norm inequality implies the quadratic form inequality
$$
  -\epsilon_{n}^{2} I_{\cH^{\circ}_{p}} \prec \phi_{p}(q^{(n)}) -
  L^{(n)}_{p} \preceq \epsilon_{n}^{2} I_{\cH^{\circ}_{p}}.
$$
In particular we have
$$
\phi_{p}(q^{(n)}) \succ L_{p}^{(n)} - \epsilon_{n}^{2}
I_{\cH_{p}^{\circ}} \succeq -\epsilon_{n}^{2} I_{\cH^{\circ}_{p}}.
$$
Spelling this condition out gives
\begin{align}
 &  \left( U_{\bR,p}[P_{\cH_{\bR,p}}M  q^{(n)}] \right)^{*}
    U_{\bR, p}[P_{\cH_{\bR,p}}M  q^{(n)}]
-\left( U_{\bS, p} [P_{\cH_{\bS, p}} q^{(n)}] \right)^{*}
 U_{\bS,p}[P_{\cH_{\bS,p}}q^{(n)}]  \notag \\
 & \qquad\qquad\qquad \succ - \epsilon_{n}^{2} I_{\cH^{\circ}_{p}} \text{ for
 all } p \in \bP.
 \label{ineq1}
 \end{align}

 Let now $n_{0}$ be an arbitrary nonnegative integer.  Note that
 $P_{\cH_{\bR,p}}$ commutes with $V_{\bR,p}$.  Hence, using the
 property \eqref{shiftintertwine} and the assumed shift-invariance of
 $M$, we see that
 $$
 (\widetilde V_{\bR,p})^{n_{0}} U_{\bR,p}[P_{\cH_{\bR,p}} M q^{(n)}] =
 U_{\bR,p}[P_{\cH_{\bR,p}} M (V_{\bS,p})^{n_{0}}q^{(n)}]
 $$
 and similarly
 $$
 (\widetilde V_{\bS,p})^{n_{0}} U_{\bS,p}[P_{\cH_{\bS,p}} q^{(n)}] =
  U_{\bR,p}[P_{\cH_{\bR,p}} (V_{\bS,p})^{n_{0}}q^{(n)}].
 $$
 It follows that \eqref{ineq1} continues to hold with
 $(V_{\bS,p})^{n_{0}}q^{(n)}$ in place of $q^{(n)}$.  Hence we may
 assume without loss of generality that \eqref{ineq1} holds with
 $q^{(n)}$ a unit vector having support in $[n_{0}, \infty)$.

 As was done in the proof of Lemma \ref{L:DP},  let us write
 $P_{[n_{0},N]}$ for the projection of $\ell^{2} \otimes \cX$ onto
 $\ell^{2}[n_{0}, N] \otimes \cX$, where the coefficient space
 $\cX$ is either $\cH^{\circ}_{\bS,p} = \widetilde \cH_{\bS,p}
 \otimes \cH^{\circ}_{p}$ or $\cH^{\circ}_{\bR,p} = \widetilde
 \cH_{\bR,p} \otimes \cH^{\circ}_{p}$;  in case $n_{0} = 0$, we write
 simply $P_{N}$ rather than $P_{[0,N]}$. In case the coefficient
 space is either $\widetilde \cH_{\bS,p}$ or $\widetilde
 \cH_{\bR,p}$, we write $\widetilde P_{[n_0,N]}$ and $\widetilde P_{N}$
 respectively.

 The key property of the
 sequence $\{P_{N}\}_{N \in {\mathbb Z}_{+}}$ is its strong
 convergence to the identity operator as $N \to \infty$.
 In particular $P_{N} q^{(n)} \to q^{(n)}$ in $\cH_{\bS}$-norm as $N
 \to \infty$.  It follows that $U_{\bS,p}[P_{\cH_{S,p}} P_{N}
 q^{(n)}] \to U_{\bS,p}[P_{\cH_{\bS,p}} q^{(n)}]$ in
 Hilbert-Schmidt norm as $N \to \infty$.  But since $\dim
 \cH^{\circ}_{p} < \infty$, this is enough to conclude that actually
 $U_{\bS,p}[P_{\cH_{S,p}} P_{N} q^{(n)}]$ converges to $U_{\bS,p}[P_{\cH_{\bS,p}}
  q^{(n)}]$
 in operator norm as $N \to \infty$. Similarly
 $U_{\bR,p}[P_{\cH_{\bR,p}}M P_{N} q^{(n)}]$ converges to
 $U_{\bR,p}[P_{\cH_{\bR,p}}M  q^{(n)}]$ in operator norm as $N
 \to \infty$.  We may thus arrange that \eqref{ineq1} holds with
 $P_{N} q^{(n)}$ in place of $q^{(n)}$, in other words, for a fixed $n_0$ we may assume without loss of generality that $q^{(n)}$ has support
 in $[n_{0}, N]$ for some sufficiently large integer $N>n_0$.

 From the estimate
 \begin{align*}
     & \| P_{N} P_{\cH_{\bR,p}} M P_{N} q^{(n)} - P_{\cH_{\bR, p}} M
     q^{(n)} \|  \\
     & \qquad\qquad \le \| P_{N} P_{\cH_{\bR,p}} M (P_{N} - I) q^{(n)} \|
     + \| (P_{N} - I) P_{\cH_{\bR,p}} M q^{(n)} \|
 \end{align*}
 coupled with the strong convergence of $P_{N}$ to the identity operator as $N
 \to \infty$,  we see that $P_{N} P_{\cH_{\bR,p}} M P_{N} q^{(n)}$
 converges to  $P_{\cH_{\bR, p}} M q^{(n)}$ in $\cH_{\bR,p}$-norm.
 It follows that $U_{\bR,p}[P_{N} P_{\cH_{\bR,p}} M P_{N} q^{(n)}]$
 converges to $U_{\bR,p}[P_{\cH_{\bR, p}} M q^{(n)}]$ in operator
 norm.  Hence be taking $N$ sufficiently large, \eqref{ineq1} can be
 adjusted to have the form
 \begin{align}
    & \left( U_{\bR,p}[P_{[n_{0},N]} P_{\cH_{\bR,p}} M P_{[n_{0},N]}
    q^{(n)}]\right)^{*}
  U_{\bR,p}[P_{[n_{0},N]} P_{\cH_{\bR,p}} M P_{[n_{0},N]} q^{(n)}]  \notag \\
 & \qquad\qquad   -\left( U_{\bS,p}[P_{\cH_{\bS,p}} P_{[n_{0},N]} q^{(n)}]\right)^{*}
   U_{\bS,p}[P_{\cH_{\bS,p}} P_{[n_{0},N]} q^{(n)}] \succ -
   \epsilon_{n}^{2} I_{\cH^{\circ}_{p}}.
   \label{ineq2}
 \end{align}
 Furthermore, since $P_{[n_{0},N]}M P_{[n_{0},N]}=P_N M P_{[n_{0},N]}$, by the shift invariance of $M$, we have
\begin{align*}
    & \|(I - P_{[n_{0},N]}) M P_{[n_{0},N]} q^{(n)} \|=\|(I - P_N) M P_{[n_{0},N]} q^{(n)} \|
 \end{align*}
 combined with the strong convergence of $P_{N}$ to the identity operator as $N \to
 \infty$ tells us that we may also arrange that
 \begin{equation}  \label{ineq1'}
     \| (I - P_{[n_{0},N]}) M P_{[n_{0},N]} q^{(n)} \| < \epsilon_{n}
 \end{equation}
 by taking $N$ sufficiently large.  By rescaling and taking $N$ still
 larger if necessary, we can assume without loss of generality that
 $P_{[n_{0},N]} q^{(n)}$ is a unit vector.  By redefining $q^{(n)}$
 to be $P_{[n_0,N]}q^{(n)}$, we arrive at the following normalization:
 {\em for any $n_{0} \in {\mathbb
 Z}_{+}$, by taking $N \in {\mathbb Z}_{+}$ sufficiently large, we
 can find a unit vector $q^{(n)} \in \ell^{2} \otimes
 \cH^{\circ}_{\bS}$ with support in $[n_{0}, N]$ so that
  \begin{align}
    & \left( U_{\bR,p}[P_{[n_{0},N]} P_{\cH_{\bR,p}} M q^{(n)}]\right)^{*}
  U_{\bR,p}[P_{[n_{0},N]} P_{\cH_{\bR,p}} M q^{(n)}]  \notag \\
 & \qquad\qquad   -\left( U_{\bS,p}[P_{\cH_{\bS,p}} q^{(n)}]\right)^{*}
   U_{\bS,p}[P_{\cH_{\bS,p}} q^{(n)}] \succ -
   \epsilon_{n}^{2} I_{\cH^{\circ}_{p}}.
   \label{ineq2'}
 \end{align}
 as well as}
  \begin{equation}  \label{ineq1''}
     \| (I - P_{[n_{0},N]}) M  q^{(n)} \| < \epsilon_{n}.
 \end{equation}

 Let us apply the Douglas lemma to the inequality \eqref{ineq2'} (see
 the discussion immediately preceding Proposition \ref{P:Douglas});
 the result is the existence of a contraction operator
 $$
 \begin{bmatrix} X^{(n)}_{p} & Y^{(n)}_{p} \end{bmatrix} \colon
     \begin{bmatrix}  \ell^{2}([n_0,N]) \otimes \widetilde \cH_{\bR,p}
	 \\ \cH^{\circ}_{p} \end{bmatrix} \to
\ell^{2}([n_{0},N]) \otimes \widetilde \cH_{\bS,p}
$$
so that
$$ X_{p}^{(n)} U_{\bR,p}[ P_{[n_{0},N]} P_{\cH_{\bR,p}} M
P_{[n_{0},N]} q^{(n)}] + \epsilon_{n} Y^{(n)}_{p} =
U_{\bS,p}[P_{\cH_{\bS,p}} P_{[n_{0},N]} q^{(n)}].
$$
Hence, as a consequence of property \eqref{intertwine}, we get
\begin{equation}  \label{Douglas-solve}
U_{\bS,p}\left[ \left( X^{(n)}_{p} \otimes I_{\cH^{\circ}_{p}}
\right) P_{[n_{0},N]} P_{\cH_{\bR,p}} M P_{[n_{0},N]} q^{(n)} -
P_{\cH_{\bS,p}} P_{[n,N]} q^{(n)} \right] = \epsilon_{n} Y^{(n)}_{p}.
\end{equation}

As $\| Y_{p}^{(n)} \| \le 1$, it follows that
$$
{\rm tr}\, (Y_{p}^{(n) *} Y_{p}^{(n)} ) = \sum_{j \in J} \| Y_{p}
e^{(p)}_{j} \|^{2} \le \dim \cH^{\circ}_{p}
$$
where we let $\{e_{j}^{(p)} \colon j \in J\}$ be any orthonormal
basis for $\cH^{\circ}_{p}$.  Hence we see that $Y_{p}^{(n)}$ has
Hilbert-Schmidt norm $\| Y_{p}^{(n)} \|_{\cC_{2}(\cH_{p}^{\circ},
\ell^{2} \otimes \widetilde \cH_{\bS,p})}$ at most $( \dim
\cH^{\circ}_{p})^{1/2}$.  As $U_{\bS,p}$ is unitary from
$\cH_{\bS,p}$ to $\cC_{2}(\cH^{\circ}_{p}, \ell^{2} \otimes
\widetilde \cH_{\bS,p})$, we see from equality
\eqref{Douglas-solve} that
\begin{equation}   \label{ineq3}
\| (X^{(n)}_{p} \otimes I_{\cH^{\circ}_{p}}) P_{[n_{0},N]}
P_{\cH_{\bR,p}} M q^{(n)} - P_{\cH_{\bS,p}} q^{(n)} \| < \epsilon_n
\cdot (\dim \cH^{\circ}_{p})^{1/2}.
\end{equation}
If we set $\De^{(n)} =
{\rm diag}_{p \in \bP} P_{[n_{0}, N]}
\left( X^{(n)}_{p} \otimes I_{\cH^{\circ}_{p}} \right) P_{[n_{0}, N]}$,
then we see that $\De^{(n)}$ has the correct block-diagonal
structure to be an element of the structure
$\bDelta_{\overline{\bG}}$ and furthermore $\| \De^{(n)}\| \le
1$ since $\| X^{(n)}_{p} \otimes I_{\cH^{\circ}_{p}} \| \le 1$ for
each $p$, i.e., $\De^{(n)}  \in \overline\cB
\bDelta_{\overline{\bG}}$.
Furthermore, from \eqref{ineq3} we see that
\begin{equation}  \label{ineq4}
    \| (I - \De^{(n)} M) q^{(n)} \| < \epsilon_{n} (\sum_{p
    \in \bP} \dim \cH^{\circ}_{p})^{1/2}.
\end{equation}

We now exploit the arbitrariness of $n_{0}$ in the preceding
analysis.  Proceeding inductively, we may take the support of
$q^{(n)}$ to be contained in an interval of the form $[t_{n},
t_{n+1}) \subset {\mathbb Z}_{+}$ with $t_{0} = 0$ such that these
intervals form a complete partition of ${\mathbb Z}_{+}$.  Now set
$$
  \De = \sum_{n=0}^{\infty} \De^{(n)} P_{[t_{n}, t_{n+1})}.
$$
Then it is easily seen that $\De \in \overline{\cB}
\bDelta_{\overline{\bG}}$.  When we apply $I - \De M$ to $q^{(n)}$ and
estimate the norm, we get
\begin{align}
    \| (I - \De M) q^{(n)} \| & = \| (I - \De \{ P_{[t_{n}, t_{n+1})} +
    (I - P_{[t_{n}, t_{n+1})}) \} M q^{(n)} \| \notag \\
    & = \| (I - \De^{(n)} M) q^{(n)} - \De (I - P_{[t_{n},
    t_{n+1})}) M q^{(n)} \|  \\
    & \le \| (I - \De^{(n)}M) q^{(n)} \| + \| (I - P_{t_{n},
    t_{n+1})}) M q^{(n)} \|\notag  \\
    & < \epsilon_{n} \cdot ( \sum_{p \in \bP} \dim
    \cH^{\circ}_{p})^{1/2} + \epsilon_{n}
    \label{estimate}
\end{align}
by \eqref{ineq4} and \eqref{ineq1''}.  As each $q^{(n)}$ is a  unit
vector and $\epsilon_{n} \cdot ( \sum_{p \in \bP} \dim \cH^{\circ}_{p})^{1/2}
+ \epsilon_{n} \to 0$ as $n \to \infty$,  we conclude that $I - \De M$ is not bounded below
and hence cannot be boundedly invertible, in contradiction to the
assumption that $\widehat \mu_{\bDelta_{\overline{\bG}}}(M) < 1$.
This completes the verification of Step 1.

\paragraph{\bf Verification of Step 2:}   The proof is modeled on Lemma 8.11
in \cite{DP} and follows the original idea of Megretski and Treil \cite{MT}.

Let $h, \widetilde h \in \cH_{\bS}$ with $\| h \| = \| \widetilde h
\| = 1$ and let $\alpha \in (0,1)$.  We shall prove that $\alpha
\phi_{p}(h) + (1 - \alpha) \phi_{p}(\widetilde h) \in
\overline{\nabla}$ for each $p \in \bP$.  The convexity of the set
$\overline{\nabla}$ then follows via a straightforward continuity
argument.

For each $n \in {\mathbb Z}_{+}$, set $h_{n} = \sqrt{\alpha} h +
\sqrt{1 - \alpha} V_{\bS}^{n} \widetilde h \in \cH_{\bS}$.  Then
\begin{align*}
    \| h_{n}\|^{2} &  = \alpha \| h \|^{2} + (1 - \alpha) \| V_{\bS}^{n}
    \widetilde h \|^{2} + 2 \sqrt{\alpha (1 - \alpha)}\, {\rm Re}\,
    \langle h, V_{\bS}^{n} \widetilde h \rangle \\
    & = \alpha + (1 - \alpha) + 2 \sqrt{(\alpha (1 - \alpha)} \,{\rm
    Re} \, \langle V_{\bS}^{*n} h, \widetilde h \rangle \\
    & = 1 + 2 \sqrt{ \alpha (1 - \alpha)}\,  {\rm Re}\, \langle
    V_{\bS}^{*}n h, \widetilde h \rangle.
\end{align*}
Since $V_{\bS}^{*n}$ converges strongly (hence also weakly) to $0$,
we conclude that
\begin{equation}   \label{norm1}
    \| h_{n}\|^{2} \to 1 \quad\text{ as }\quad n \to \infty.
\end{equation}

Next, writing out $h_{n} = \sqrt{\alpha} h +\sqrt{1 - \alpha} V_{\bS}^{n} \widetilde h$ and using the linearity of $U_{\bR,p}$ and $U_{\bS,p}$ we observe that
\begin{align*}
  &  \phi_{p}(h_{n})  = \left(U_{\bR,p}[P_{\cH_{\bR,p}}M  h_{n}]\right)^{*}
    U_{\bR,p}[P_{\cH_{\bR,p}}M h_{n}] -
\left( U_{\bS,p}[P_{\cH_{\bS,p}} h_{n}]\right)^{*}
U_{\bS,p}[P_{\cH_{\bS,p}} h_{n}] \\
& \qquad\qquad = \alpha \phi_{p}(h) + (1- \alpha) \phi_{p}(V_{\bS}^{n} \widetilde
h)  + 2 \sqrt{\alpha (1 - \alpha)} \cdot \\
&   \cdot {\rm Re}\,
\left( \left(U_{\bR,p}[P_{\cH_{\bR,p}} M h]\right)^{*}  U_{\bR,p}[P_{\cH_{\bR,p}}
M V_{\bS}^{n} \widetilde h]
-\left( U_{\bS,p}[P_{\cH_{\bS,p}} h]\right)^{*}
U_{\bS,p}[P_{\cH_{\bS,p}} V_{\bS}^{n} \widetilde h] \right).
\end{align*}
A consequence of the intertwining property \eqref{shiftintertwine}
and the shift-invariance of $M$ is that in fact
$$
\phi_{p}(V_{\bS}^{n} \widetilde h) = \phi_{p}(\widetilde h).
$$
Thus in fact we have
\begin{align}
    & \phi_{p}(h_{n}) =
 \alpha \phi_{p}(h) + (1- \alpha) \phi_{p}(\widetilde h)  + 2
 \sqrt{\alpha (1 - \alpha)} \cdot  \notag \\
&  \cdot {\rm Re}\,
\left( \left(U_{\bR,p}[P_{\cH_{\bR,p}} M h]\right)^{*}
U_{\bR,p}[P_{\cH_{\bR,p}} M V_{\bS}^{n} \widetilde h]
-\left( U_{\bS,p}[P_{\cH_{\bS,p}} h]\right)^{*}
U_{\bS,p}[P_{\cH_{\bS,p}} V_{\bS}^{n} \widetilde h] \right).
\label{phipconv}
\end{align}
We claim that the cross terms tend to zero (in trace-class norm) as $n \to \infty$.  A
sample term to check is
\begin{equation}   \label{check1}
 \left(U_{\bR,p}[P_{\cH_{\bR,p}} M h]\right)^{*}
 U_{\bR,p}[P_{\cH_{\bR,p}} M V_{\bS}^{n} \widetilde h]  \to 0 \text{ as } n \to
 \infty.
\end{equation}
We again use the shift invariance of $M$ and the intertwining property
\eqref{shiftintertwine} to see that
\begin{align*}
 & \left(  U_{\bR,p}[P_{\cH_{\bR,p}} M h] \right)^{*}  U_{\bR,p}[P_{\cH_{\bR,p}} M
    V_{\bS}^{n} \widetilde h]  =
\left( U_{\bR,p}[P_{\cH_{\bR,p}} M h] \right)^{*} \widetilde
V_{\bR}^{n} U_{\bR,p}[P_{\cH_{\bR,p}} M \widetilde h]
\\ & \qquad \qquad
=\left(  \widetilde V_{\bR}^{* n} U_{\bR,p}[P_{\cH_{\bR,p}} M h] \right)^{*}
     U_{\bR,p}[P_{\cH_{\bR,p}} M \widetilde h] \\
     &  \qquad \qquad = \left( U_{\bR,p}[ V_{\bR}^{* n} P_{\cH_{\bR,p}} M h] \right)^{*}
 U_{\bR,p}[  P_{\cH_{\bR,p}} M \widetilde h]
\end{align*}
The fact that $V_{\bR,p}^{* n} P_{\cH_{\bR,p}} M h \to 0$ in
$\cH_{\bR,p}$ implies that $U_{\bR,p}[ V_{\bR,p}^{*n} P_{\cH_{\bR,p}}
M h] \to 0$ in Hilbert-Schmidt norm, and hence \eqref{check1} now
follows.  A similar calculation shows that
$$
 \left( U_{\bS, p}[ P_{\bS,p} h] \right)^{*}  U_{\bS, p}[ P_{\bS,p} V_{\bS}^{n} \widetilde h]
 \to 0 \text{ as } n \to \infty.
$$
From \eqref{phipconv} we now read off
$$
\phi_{p}(h_{n}) \to \alpha \phi_{p}(h) + (1 - \alpha)
\phi_{p}(\widetilde h).
$$
As we observed already in \eqref{norm1} that $\|h_{n}\| \to 1$ as $n
\to \infty$, we see that we also have
$$
 \phi_{p}\left( \frac{h_{n}}{\| h_{n} \|}\right) =
 \frac{1}{\|h_{n}\|^{2}} \phi(h_{n}) \to \alpha \phi_{p}(h) + (1 -
 \alpha)  \phi_{p}(\widetilde h) \text{ as } n \to \infty.
$$
This exhibits the convex combination $\alpha \phi_{p}(h) + (1 -
\alpha) \phi_{p}(\widetilde h)$ of two elements of $\nabla$ as an
element of $\overline{\nabla}$ and completes the verification of Step
2.

\smallskip
The proof of Theorem \ref{T:I'} is now complete once one observes
that the combined results of Steps 1 and 2 lead immediately to the
validity of the implication \eqref{toshow}.
\end{proof}

\begin{remark} \label{R:nonsquare}  {\em  We note the following more
    general version of Theorem \ref{T:I}.

    \begin{theorem}  \label{T:Paganini}
	Let $\overline{\bG}$ and $\bG^{\circ}$ be as in Theorem
	\ref{T:I} with the number of components of $\bG^{\circ}$
	finite and with all coefficient Hilbert spaces
	$\cH^{\circ}_{p}$ finite-dimensional. Let $\cK$ be a fixed
	separable infinite-dimensional Hilbert space (e.g., $\cK =
	\ell^2$).  Then the following
	stabilizability and detectability results hold.
	
\begin{enumerate}
    \item  Suppose that $(A,B)$ is an input-pair of the form
    $$
      \begin{bmatrix} A & B \end{bmatrix} \colon \begin{bmatrix}
	  \cH^{\circ}_{\bS} \\ \cU \end{bmatrix} \to \cH^{\circ}_{\bR}
    $$
    Then the following conditions are equivalent:
    \begin{enumerate}
	
\item 	The operator
	$$ \begin{bmatrix} I - \Delta (I \otimes A)  & I \otimes B
	\end{bmatrix} \colon \begin{bmatrix} \cK \otimes \cH^{\circ}_{\bS} \\ \cK
	\otimes \cU \end{bmatrix} \to \cK \otimes \cH^{\circ}_{\bR}
	$$
is boundedly left invertible for all $\Delta \in \rho \overline{\cB}
\bDelta_{\overline{\bG}}$ for some $\rho > 1$.

\item There exist positive definite operators $\Gamma^{\circ}_{p}
\succ 0$ on $\cH^{\circ}_{p}$ so that
$$
  A \left( \bigoplus_{p \in \bP} (I_{\cH_{\bS,p}} \otimes
  \Gamma^{\circ}_{p})  \right) A^{*} - \left( \bigoplus_{p \in \bP}
  I_{\cH_{\bR,p}} \otimes \Gamma^{\circ}_{p} \right) - B B^{*}
  \preceq 0.
$$

\item There exist a feedback operator $F \colon \cH_{\bS} \to \cU$ so
that $\mu_{\overline{\bG}}(I_{\cK} \otimes (A + BF)) < 1$, i.e., for
some $\rho > 1$ the operator $I - \Delta (I \otimes (A + B F))$ is
boundedly invertible for each $\Delta \in \overline{\cB}
\bDelta_{\overline{\bG}}$.
\end{enumerate}

\item Suppose that $(C,A)$ is an output-pair of the form
$$
   \begin{bmatrix} A \\ C \end{bmatrix}  \colon \cH^{\circ}_{\bS} \to
       \begin{bmatrix} \cH^{\circ}_{\bR} \\ \cY \end{bmatrix}.
$$
Then the following conditions are equivalent:
\begin{enumerate}
    \item  The operator
$$
\begin{bmatrix} I - \Delta(I \otimes A) \\ I \otimes C \end{bmatrix}
    \colon \cK \otimes \cH^{\circ}_{\bS} \to \begin{bmatrix}
   \cK \otimes  \cH^{\circ}_{\bS} \\ \cK \otimes \cY \end{bmatrix}
$$
is boundedly left invertible for all $\Delta \in \rho \overline{\cB}
\bDelta_{\overline{\bG}}$ for some $\rho > 1$.

\item There exist positive definite operators $\Gamma^{\circ}_{p}
\succ 0$ on $\cH^{\circ}_{p}$ so that
$$
  A^{*} \left( \bigoplus_{p \in \bP} (I_{\cH_{\bR,p}} \otimes
  \Gamma^{\circ}_{p})  \right) A - \left( \bigoplus_{p \in \bP}
  I_{\cH_{\bS,p}} \otimes \Gamma^{\circ}_{p} \right) - C^{*}C
  \prec 0.
$$

\item There exists an output injection $L \colon \cY \to \cH_{\bR}$
so that $\mu_{\overline{\bG}}(I_{\cK} \otimes (A + LC)) < 1$, i.e., for
some $\rho > 1$ the operator $I - \Delta (I \otimes (A + LC))$ is
boundedly invertible for each $\Delta \in \rho \overline{\cB}
\bDelta_{\overline{\bG}}$.
\end{enumerate}
\end{enumerate}
\end{theorem}

For the simple multiplicity case ($\cH^{\circ}_{p} = {\mathbb C}$ for
all $p$), details of this result can be found in \cite{Paganini}; a
nice summary (with no proofs) is in \cite{ZDG}. We expect that either
of the proofs of Theorem \ref{T:I} presented here can be adapted to
arrive at the more general formulation in Theorem \ref{T:Paganini}.
We refer also to \cite{BtH} for additional information and perspective.
}\end{remark}

\begin{remark} \label{R:transfunc}  {\em
    For the case where $M = I_{\ell^{2}} \otimes M^{\circ}$ where
    $M^{\circ}$ is an operator between the finite-dimensional spaces $\cH^{\circ}_{\bS}$ to
    $\cH^{\circ}_{\bR}$, the Linear Operator Inequality
    \eqref{muhatcond} reduces to the finite-dimensional Linear Matrix
    Inequality
\begin{equation}   \label{muhatcondLMI}
    M^{\circ *} \left( \bigoplus_{p \in \bP} I_{\widetilde
    \cH_{\bR,p}} \otimes \Gamma^{\circ}_{p} \right) M^{\circ}
- \left( \bigoplus_{p \in \bP} I_{\widetilde \cH_{\bS,p}} \otimes
\Gamma^{\circ}_{p} \right) \prec 0.
\end{equation}
In case $M$ is a shift-invariant operator from $\ell^{2}
\otimes \cH^{\circ}_{\bS}$ to $\ell^{2} \otimes \cH^{\circ}_{\bR}$,
there appears no reason for the LOI \eqref{muhatcond} to collapse to
an LMI like \eqref{muhatcondLMI} in general.  However, if we assume
that $M$ is given via convolution with a distribution having $Z$-transform equal to
a rational matrix function $\widehat M(\la)$ having state-space realization
$$
   \widehat M(\la)  = D + \la C (I - \la A)^{-1} B
$$
with $A$ stable (the spectrum $\sigma(A)$ of $A$ is inside the unit disk
${\mathbb D}$)
where the system matrix
$$
 \bU = \begin{bmatrix} A  & B \\ C & D \end{bmatrix} \colon
 \begin{bmatrix} \cX \\ \widetilde \cH_{\bS} \end{bmatrix} \to \begin{bmatrix} \cX \\
 \widetilde \cH_{\bR} \end{bmatrix}
$$
is finite-dimensional,  then it is possible to convert the LOI
\eqref{muhatcond} to an LMI condition as follows.  Rewrite the LOI
\eqref{muhatcond} as
\begin{equation}   \label{rewrite1}
\left\| \left( \bigoplus_{p \in \bP} I_{\ell^{2}} \otimes (I_{\widetilde
\cH_{\bR,p}} \otimes \Gamma^{\circ}_{p})^{1/2} \right) M
\left( \bigoplus_{p \in \bP} I_{\ell^{2}} \otimes (I_{\widetilde
\cH_{\bS,p}} \otimes \Gamma^{\circ}_{p})^{1/2} \right)^{-1/2} \right\| < 1.
\end{equation}
After applying the $Z$-transform to move to the frequency domain, we
see from \eqref{rewrite1} that
\begin{equation}\label{rationalDscale}
 \sup_{\la \in {\mathbb D}} \left\| \left( \bigoplus_{p \in
 \bP}(I_{\widetilde \cH_{\bR,p}} \otimes \Gamma^{\circ}_{p})^{1/2}
 \right) \widehat M(\la) \left( \bigoplus_{p \in \bP} (I_{\widetilde \cH_{\bS,p}}
 \otimes \Gamma^{\circ}_{p})^{-1/2} \right) \right\| < 1.
\end{equation}
As we are assuming that $A$ is stable, the standard strict Bounded
Real Lemma implies that there is a positive-definite $X \succ 0$ on
$\cX$ so such that
\begin{equation}  \label{tildeBRL}
\begin{bmatrix} \widetilde A^{*} & \widetilde C^{*} \\ \widetilde
B^{*} & \widetilde D^{*} \end{bmatrix} \begin{bmatrix} X & 0 \\ 0 & I
\end{bmatrix} \begin{bmatrix} \widetilde A & \widetilde B \\
\widetilde C & \widetilde D \end{bmatrix} - \begin{bmatrix} X & 0 \\
0 & I \end{bmatrix} \prec 0
\end{equation}
where we have set
$$
  \begin{bmatrix} \widetilde A & \widetilde B \\ \widetilde C &
      \widetilde D \end{bmatrix}  =
   \begin{bmatrix} I & 0 \\ 0 & (I_{\widetilde \cH_{\bR,p}} \otimes
       \Gamma^{\circ}_{p})^{1/2} \end{bmatrix}
 \begin{bmatrix} A & B  \\ C & D \end{bmatrix}
  \begin{bmatrix} I & 0 \\ 0 & ( I_{\widetilde \cH_{\bR,p}} \otimes
      \Gamma^{\circ}_{p})^{-1/2} \end{bmatrix}.
$$
The condition \eqref{tildeBRL} in turn can be rewritten in the form
\begin{equation}   \label{LMIreduced}
    \begin{bmatrix} A^{*} & C^{*} \\ B^{*} & D^{*} \end{bmatrix}
\begin{bmatrix} X & 0 \\ 0 & \bigoplus_{p \in \bP} (I_{\widetilde
    \cH_{\bR,p}} \otimes \Gamma^{\circ}_{p}) \end{bmatrix}
 \begin{bmatrix} A & B \\ C & D \end{bmatrix}
-  \begin{bmatrix} X & 0 \\ 0 & \bigoplus_{p \in \bP} (I_{\widetilde
\cH_{\bS,p}} \otimes \Gamma^{\circ}_{p}) \end{bmatrix} \prec 0.
\end{equation}
This last condition \eqref{LMIreduced} finally gives us an LMI
equivalent to the LOI \eqref{muhatcond} for this case.  We note that
this analysis is just the discrete-time equivalent of
Proposition~8.6 in \cite{DP}.
}\end{remark}

\begin{remark} \label{R:slowlytimevarying}  {\em
 In our analysis to this point we have
    considered  structure subspaces $\bDelta \subset \cL(\ell^2 \otimes {\mathbb C}^N)$
    defined by spatial constraints (block diagonal matrix representation) without any
    dynamic constraints.  It is natural to impose some additional
    constraints involving dynamics or parameter restrictions (e.g.,
    forcing the parameters to be real)---see \cite[pages
    255--256]{DP} as well as \cite{Paganini}.
    In this extended remark we discuss some of these additional
    considerations which have been discussed in the literature.

  Consider the setting of Theorem \ref{T:I'} but with the structure
  subspace $\bDelta_{\overline{G}}$ as in
  \eqref{uncer-en2} or \eqref{uncer-en3} (with $\cK = \ell^2$) replaced by
  \begin{equation}   \label{uncer-en4}
       \bDelta_{\overline{\bG}, \textbf{TV}} = \{ W = \rm{diag}_{k=1,
      \dots, K} [W_{k} \otimes I_{\cH^{\circ}_{p_{k}}}] \colon W_{k}
      \in \cL(\ell^{2})^{n_{k} \times m_{k}} \text{ and } W_{k} V =
      V W_{k} \},
 \end{equation}
 i.e.,  $ \bDelta_{\overline{\bG}, \textbf{TV}}$ consists
 of those elements of $ \bDelta_{\overline{\bG}}$ which are also
 shift-invariant.  Then, for the case where $M$ is as in Remark
 \ref{R:transfunc} (i.e., given via multiplication by a rational
 transfer function $\widehat M(\lambda)$ after transforming to the
 frequency domain via the $Z$-transform), one can argue that
 $\mu_{ \bDelta_{\overline{\bG}, \textbf{TV}}}(M)$ is given
 by a supremum of a pointwise structured singular value for the
 matrix function $\widehat M(\zeta)$:
 \begin{equation}   \label{TVmu}
 \mu_{ \bDelta_{\overline{\bG}, \textbf{TV}}}(M) =
 \sup_{\zeta \in {\mathbb T}} \mu_{\bDelta_{\overline{\bG}}}(\widehat
 M(\zeta)).
 \end{equation}
 Indeed, this point is argued in detail in \cite[Theorem 8.22]{DP}
 for the case where the structure space $\bDelta_{\overline{\bG}}$
 has the special form \eqref{structure} (square blocks with
 only scalar blocks have higher multiplicity);  it is now straightforward
 to adapt the argument to the more general structure
 $\bDelta_{\overline{\bG}}$.  As we have already noted in Section
 \ref{S:intro}, computation of $\mu_{\bDelta_{\overline{\bG}}}(\widehat
 M(\zeta))$ at a fixed value of $\zeta$ is problematical, hence
 computation of the supremum in \eqref{TVmu} is even more so.  A natural
 upper bound for   $\mu_{ \bDelta_{\overline{\bG},
 \textbf{TV}}}(M)$ is the frequency-dependent $D$-scaling
 $$
   \widehat \mu_{ \bDelta_{\overline{\bG}, \textbf{TV}}}(M)
   : = \inf_{D} \sup_{\zeta \in {\mathbb T}}
   \{ \| (I_{\widetilde \cH_{\bR,p}} \otimes D(\zeta)) \widehat M(\zeta)
   (I_{\widetilde \cH_{\bR,p}} \otimes D(\zeta))^{-1} \|
 $$
 where the infimum can be taken over $D(\zeta)$ equal to a stable
 rational matrix function invertible on ${\mathbb T}$.  As
 discussed in Section \ref{S:intro} above, this upper bound is
 arbitrarily bad (in various technical senses) when taken at a fixed
 frequency $\zeta \in {\mathbb T}$, and hence has no chance of being
 sharp for this frequency-dependent situation.

 Poolla-Tikku \cite{PT95} provide a different perspective on this issue, by giving a robust control
 interpretation to the quantity $ \widehat \mu_{ \bDelta_{\overline{\bG}, \textbf{TV}}}(M)$. Extending the setting of \cite{PT95} to our set of structured uncertainties $\bDelta_{\overline{\bG}}$, for a positive parameter $\nu$ we let $ \bDelta_{\overline{\bG}, \nu}$ consist of
 those operators $\Delta$ in $ \bDelta_{\overline{\bG}}$
 such that
 $$
    \| V \Delta - \Delta V \| \le \nu
 $$
 where $V$ as usual is the forward shift operator on $\ell^{2}$ (of
 whatever multiplicity fits the context).  Thus operators in
 $ \bDelta_{\overline{\bG}, \nu}$ are constrained to be
 {\em slowly time-varying}, with precise amount of slowness measured
 by $\nu$ (the smaller the $\nu$ the more slow is the time variance
 with $\nu = 0$ corresponding to time-invariance and $\nu = 2 \|
 \Delta \|$ correspondence to no restriction at all).  A corollary of
 the more precise results from
 \cite{PT95}, again for the classical spacial case where $\bDelta$ is
 given by \eqref{structure}, is the following:  {\em
 $\widehat \mu_{ \bDelta_{\overline{\bG}, \textbf{TV}}}(M) < 1$ if and only
 if there is some $\nu > 0$ so that $\mu_{
 \bDelta_{\overline{\bG}, \nu}}(M) < 1$.}  As any disturbance in
 practice can be expected to have some time-variance, it is argued in
 \cite{PT95} that computation of the upper bound
 $\widehat \mu_{ \bDelta_{\overline{\bG}, \textbf{TV}}}(M)$
 makes more sense physically than the original quantity
 $\mu_{ \bDelta_{\overline{\bG}, \textbf{TV}}}(M)$.
 Followup work of Paganini \cite{PaganiniAut} (see also Chapter 3 of
 \cite{Paganini})  showed how one can incorporate time-invariant and
 time-variant blocks as well as blocks with parametric uncertainty
 simultaneously.  The paper of K\"oro\u{g}lu-Scherer \cite{KS2007}
 refines the results still further for a general block structure
 (possibly nonsquare blocks with arbitrary multiplicities) with
 preassigned bounds on the time-variation of the blocks, obtaining
 upper and lower bounds on the optimal possible performance for this
 general setting.  Much of this work (including the book \cite{DP})
 also incorporates a causality constraint on the original plant and the admissible
 perturbation operators $\Delta$.  Recent work of Scherer-K\"ose \cite{SK2012}
 analyzes the application of frequency-dependent $D$-scaling techniques to the somewhat
 more general setup of a gain-scheduled feedback configuration.
 }\end{remark}

\section{The enhanced uncertainty structure of
Bercovici-Foias-Khargonekar-Tannenbaum} \label{S:BFKT}

An alternative enhancement $\widetilde \mu_{\bDelta}(M)$ of the
structured  singular value $\mu_{\bDelta}(M)$ leading to an equality
with the upper bound $\widetilde \mu_{\bDelta}(M) = \widehat
\mu_{\bDelta}(M)$ was introduced and developed by Bercovici, Foias and
Tannenbaum in \cite{BFT1990}.  Later work with Khargonekar
\cite{BFKT, BFT1996} obtained an extension to infinite-dimensional situations.  Here we
show how the main result can be obtained as a simple adaptation of
the convexity-analysis approach  of Dullerud-Paganini.
The following result is essentially Theorem 3 from \cite{BFT1990} with a
couple of modifications:  our result is more general in that we allow
$\bDelta$ to have nonsquare blocks  and hence not a $C^{*}$-algebra;
on the other hand here we consider only the multiplicity-1 case so we are
not allowing the structure $\bDelta$ to be a general $C^{*}$-subalgebra
as in \cite{BFT1990}.  The result can also be seen to follow as a corollary of the more
general results concerning robustness with respect to mixed
linear-time-varying/linear-time-invariant structured uncertainty
(see Chapter 3 of \cite{Paganini}).

The setup is close to that of Theorem \ref{T:I} with a couple of
differences.  We let $\bG^{\circ}$ be a multiplicity-1 $M$-graph;
thus the spaces $\cH^{\circ}_{p}= {\mathbb C}$ for all $p \in \bP$.
We therefore generate the source and range coefficient spaces
$$
 \cH^{\circ}_{\bS} = \bigoplus_{p \in \bP} \cH^{\circ}_{\bS,p} \text{ where }
 \cH^{\circ}_{\bS,p} = \bigoplus_{s \colon [s] = p} {\mathbb C}, \quad
 \cH^{\circ}_{\bR} = \bigoplus_{p \in \bP} \cH^{\circ}_{\bR,p} \text{ where }
 \cH^{\circ}_{\bR,p} = \bigoplus_{r \colon [r] = p} {\mathbb C}
$$
and the structure subspace
$$
\bDelta_{\bG^{\circ}} = {\rm diag}_{p \in \bP} \cL(\cH^{\circ}_{R,p}, \cH^{\circ}_{S,p})
$$
For the enhanced structure we proceed as in Subsection
\ref{S:enhanced}, but with $\cK = \cH^{\circ}_{\bS}$.  This generates
enhanced source and range coefficient spaces
$$
\cH_{\bS} = \cH^{\circ}_{S} \otimes \cH^{\circ}_{S}, \quad
\cH_{\bS,p} = \cH^{\circ}_{\bS} \otimes \cH^{\circ}_{\bS,p},
\quad
\cH_{\bR} = \cH^{\circ}_{R} \otimes \cH^{\circ}_{R}, \quad
\cH_{\bR,p} = \cH^{\circ}_{\bR} \otimes \cH^{\circ}_{\bR,p}
$$
with the resulting enhanced structure (see \eqref{uncer-en2})
$$
\bDelta_{\overline{\bG}} = {\rm diag}_{p \in \bP}
\cL(\cH^{\circ}_{\bS} \otimes \cH^{\circ}_{\bR,p},
\cH^{\circ}_{\bS} \otimes \cH^{\circ}_{\bS,p}).
$$
Then we have the following version of Theorem 3 from \cite{BFT1990}.

\begin{theorem}  \label{T:BFT}
    Let $\bDelta_{\bG^{\circ}}$ and $\bDelta_{\overline{\bG}}$ be as
    above, assume the graph $\bG^{\circ}$ is finite and let $M^{\circ}$ be
    any operator from $\cH^{\circ}_{\bS}$ into $\cH^{\circ}_{\bR}$.
Then
$$
\widetilde \mu_{\bDelta_{\bG^{\circ}}}(M^{\circ}): =
\mu_{\bDelta_{\overline{\bG}}}(I_{\cH_{\bS}} \otimes M^{\circ}) = \widehat
\mu_{\bDelta_{\bG^{\circ}}}(M^{\circ}).
$$
\end{theorem}

\begin{proof}
    We use the identification maps
$$
\left(U_{\cH^{\circ}_{\bS}, \cH^{\circ}_{\bS}} \right)^{\top} \colon
\cH^{\circ}_{\bS} \otimes \cH^{\circ}_{\bS} \to
\cC_{2}(\cH^{\circ}_{\bS}), \quad
\left(U_{\cH^{\circ}_{\bS}, \cH^{\circ}_{\bR}} \right)^{\top} \colon
\cH^{\circ}_{\bS} \otimes \cH^{\circ}_{\bS} \to
\cC_{2}(\cH^{\circ}_{\bS}, \cH^{\circ}_{\bR})
$$
to view elements of $\cH_{\bS}$ and $\cH_{\bR}$ as being in the Hilbert
spaces of Hilbert-Schmidt operators $\cC_{2}(\cH^{\circ}_{\bS})$ and
$\cC_{2}(\cH^{\circ}_{\bS}, \cH^{\circ}_{\bR})$ from the start.  With
these identifications, the operator $I \otimes M$ becomes the operator
$L_{M} \colon \cC_{2}(\cH^{\circ}_{S}) \to \cC_{2}(\cH^{\circ}_{\bS},
\cH^{\circ}_{\bR})$ of left multiplication by $M$ and the structure
space becomes
$$
\bDelta_{\overline{\bG}} = {\rm diag}_{p \in \bP}
\cL(\cC_{2}(\cH_{\bS}^{\circ}, \cH^{\circ}_{\bR,p}),
\cC_{2}(\cH^{\circ}_{\bS}, \cH^{\circ}_{\bS,p}))$$
(note that elements of the spaces $\cL(\cC_{2}(\cH_{\bS}^{\circ}, \cH^{\circ}_{\bR,p}),
\cC_{2}(\cH^{\circ}_{\bS}, \cH^{\circ}_{\bS,p}))$ are not required to
be left multipliers).  For each $p \in \bP$ we define maps
$$
\phi_{p} \colon \cH_{\bR,p} : = \cC^{2}(\cH^{\circ}_{\bS},
\cH^{\circ}_{\bR,p}) \to {\mathbb C}
$$
by
\begin{equation}   \label{phip'}
    \phi_{p} \colon h \mapsto {\rm tr}\left( (M^{\circ *}
    P_{\cH^{\circ}_{\bR,p}} M^{\circ} - P_{\cH^{\circ}_{\bS,p}}) h
    h^{*} \right).
\end{equation}

We set
\begin{equation}   \label{NablaPi'}
    \nabla = \{ (\phi_{p}(h))_{p \in \bP} \colon \| h \|_{\cH_{\bS}}
    = 1 \}, \quad   \Pi = \{ (r_{p})_{p \in \bP} \colon r_{p} \in
    {\mathbb R} \text{ with } r_{p} \ge 0\}.
\end{equation}

Then we have the following analogue of Lemma \ref{L:1}.

\begin{lemma}   \label{L:1'}  Assume that
    $\mu_{\bDelta_{\overline{\bG}}}(L_{M}) < 1$.  Then, with $\nabla$
    and $\Pi$ as in \eqref{NablaPi'},
$$
\nabla \cap \Pi  = \emptyset.
$$
\end{lemma}

\begin{proof}
    We proceed by contradiction.  Suppose that there is an $h \in
    \cH_{\bS} = \cC_{2}(\cH^{\circ}_{\bS})$ with norm $1$ and
    $\phi_{p}(h) > 0$ for all $p$.  This means that
 $$
 \| P_{\cH^{\circ}_{\bR,p}} M^{\circ} h
 \|^{2}_{\cC_{2}(\cH^{\circ}_{\bS}, \cH^{\circ}_{\bR,p}} - \|
 P_{\cH^{\circ}_{\bS,p}} h \|^{2}_{\cC_{2}(\cH^{\circ}_{\bS,
 \cH^{\circ}_{\bS,p}})} \ge 0
 $$
 for all $p \in \bP$.  Here it is understood that the projections and
 $M^{\circ}$  act via left multiplication.  By the standard Douglas
 lemma \cite{Douglas}, there is a contraction operator $\Delta_{p}$
 from $\cC_{2}(\cH^{\circ}_{\bS}, \cH^{\circ}_{\bR,p})$ to
 $\cC_{2}(\cH^{\circ}_{\bS}, \cH^{\circ}_{\bS,p})$ (not necessarily a left multiplier)
 so that $\Delta_{p}[ P_{\cH^{\circ}_{\bR,p}} M^{\circ} h ] =
 P_{\cH^{\circ}_{\bS,p}} h$.  It we set $\Delta = {\rm diag}_{p \in
 \bP} \Delta_{p}$, then $\Delta$ is in $\overline{\cB}
 \bDelta_{\overline{\bG}}$ and $(I - \Delta L_{M}) h = 0$.
 Hence $I - \Delta L_{M}$ is not invertible, contrary to the
 assumption that $\mu_{\bDelta_{\overline{\bG}}}(L_{M}) < 1$.
 This completes the proof of the lemma.
    \end{proof}

 The next lemma (the analogue of Lemma \ref{L:2}) gives the
 connection between $\widehat
 \mu_{\bDelta_{\bG^{\circ}}}(M^{\circ}) < 1$ and the quadratic
 forms $\phi_{p}$.

 \begin{lemma}   \label{L:2'}  The condition $\widehat
     \mu_{\bDelta_{\bG^{\circ}}}(M^{\circ}) < 1$ holds if and only if
     either of the following conditions holds:
     \begin{enumerate}
	 \item  There are positive numbers $r_{p}$ ($p \in \bP$) so
	 that
	 $$
	 \sum_{p \in \bP} \gamma_{p} \phi_{p}(h) \le -\epsilon \| h
	 \|^{2}_{\cC_{2}(\cH^{\circ}_{\bS}}
	 $$
for all $h \in \cC_{2}(\cH^{\circ}_{\bS})$.

\item The sets $\nabla$ and $\Pi$ (see \eqref{NablaPi'}) are strictly
separated:  there exist real numbers $\gamma_{p} \in {\mathbb R}$ and
numbers $\alpha < \beta$ so that
$$
\sum_{p \in \bP} {\rm Re}\, \gamma_{p} k_{p} \le \alpha < \beta  \le
\sum_{p \in \bP} {\rm Re}\, \gamma_{p} r_{p} \text{ for } (k_{p})_{p
\in \bP} \in \nabla \text{ and } (r_{p})_{p \in \bP} \in \Pi.
$$
\end{enumerate}
 \end{lemma}

 \begin{proof}
     The condition $\widehat \mu_{\bDelta_{\bG^{\circ}}}(M^\circ)$ can be
     expressed as:  {\em there exist numbers $\gamma_{p} > 0$ ($p \in
     \bP$) so that}
 $$
 M^{*} \left( \bigoplus_{p \in \bP}  \gamma_{p}
 I_{\cH^{\circ}_{\bR,p}} \right) M - \left( \bigoplus_{p \in \bP}
 \gamma_{p} I_{\cH_{\bS,p}}\right) \prec 0.
 $$
 We rewrite this as the higher multiplicity quadratic-form condition
 $$
 \sum_{p \in \bP} \gamma_{p} \left( \| P_{\cH^{\circ}_{\bR,p}} M
 h\|^{2}_{\cC_{2}(\cH^{\circ}_{\bS}, \cH^{\circ}_{\bR,p})}  -
 \| P_{\cH^{\circ}_{\bS,p}} h \|^{2}_{\cC_{2}(\cH_{\bS}^{\circ},
 \cH^{\circ}_{\bS,p})} \right) < -\epsilon^{2} \| h \|^{2}
 $$
 for some $\epsilon > 0$.  This can be manipulated to the equivalent
 form
 $$
 \sum_{p \in \bP} \gamma_{p} {\rm tr} \left(
 (M^{*}P_{\cH^{\circ}_{\bR,p}} M - P_{\cH^{\circ}_{\bS,p}}) h H^{*}
 \right) = \sum_{p \in \bP} \gamma_{p} \phi_{p}(h) < -\epsilon^{2} \| h \|^{2}
 $$
 verifying (1).  The equivalence of (1) and (2) proceeds just as in
 the proof of Lemma \ref{L:2}.
  \end{proof}

  To complete the proof, following the same strategy as in the proof
  of Theorem~\ref{T:I'}, a Hahn-Banach separation theorem
  (specifically, Theorem 3.4 part (b) in \cite{Rudin}) enables to
  complete the proof if we can show the strengthened version of the
  result of Lemma \ref{L:1'}, namely:
  \begin{equation}    \label{toshow'}
      \mu_{\bDelta_{\overline{\bG}}}(L_{M^{\circ}}) < 1\quad \Rightarrow \quad
      \overline{\rm co}\, \nabla  \cap \Pi = \emptyset.
  \end{equation}
 In the present setting, the space $\cH_{\bS} =
 \cC_{2}(\cH^{\circ}_{\bS})$ is finite-dimensional and hence has
 compact unit ball.  A standard continuity argument then
 implies that $\nabla = \overline{\nabla}$ is in fact a closed subset
 in ${\mathbb R}^{\bP}$.  Thus the only remaining piece to show is
 that $\nabla$ itself is already convex.  This follows from the
 following lemma.

 \begin{lemma} \label{L:3'}
     The set $\nabla$ given by \eqref{NablaPi'} is convex.
 \end{lemma}

 \begin{proof}
     Suppose that $h$ and $\widetilde h$ are two unit-norm elements
     of $\cC_{2}(\cH^{\circ}_{\bS})$ and $0 < \alpha < 1$.  We must
     find a unit-norm $k \in \cC_{2}(\cH^{\circ}_{\bS})$ so that
     $$
     \phi_{p}(k) = \alpha \phi_{p}(h) + (1 - \alpha)
     \phi_{p}(\widetilde h) \text{ for all } p \in \bP.
     $$
Towards this end we observe that
\begin{equation}   \label{observe}
\alpha \phi_{p}(h) + (1 - \alpha) \phi_{p}(\widetilde h) =
{\rm tr} \left( (M^{\circ *} P_{\cH^{\circ}_{\bR,p}} M^{\circ} -
P_{\cH^{\circ}_{\bS,p}}) (\alpha h h^{*} + (1 - \alpha) \widetilde h
\widetilde h^{*}) \right).
\end{equation}
Note that the operator $\Upsilon: = \alpha h h^{*} + (1 - \alpha) \widetilde h
\widetilde h^{*}$ is a trace class operator of rank at most $\dim
\cH^{\circ}_{\bS}$.  Therefore we may factor $\Upsilon$ in the form
$\Upsilon = k k^{*}$ where $k \in \cC_{2}(\cH^{\circ}_{\bS})$.
Furthermore
\begin{align*}
\| k \|^{2}_{\cC_{2}(\cH^{\circ}_{\bS}} & = {\rm tr}\, (k^{*}k)
= {\rm tr}\, (\alpha h h^{*} + (1- \alpha) \widetilde h \widetilde
h^{*}) = \alpha {\rm tr}\, )h h^{*}) + (1 - \alpha) {\rm tr}
(\widetilde h \widetilde h^{*})   \\
& = \alpha \| h \|^{2} + (1 - \alpha) \| \widetilde h \|^{2} = 1
\end{align*}
so $k$ also has unit norm.  Finally, from \eqref{observe} we read off
$$
  \phi_{p}(k) = {\rm tr} \left( (M^{\circ *} P_{\cH^{\circ}_{\bR,p}} M^{\circ} -
P_{\cH^{\circ}_{\bS,p}}) k k^{*}\right) = \alpha \phi_{p}(h) + (1 - \alpha)
\phi_{p}(\widetilde h)
$$
as wanted. This completes the proof of the lemma.
\end{proof}

Putting all these pieces together completes the proof of  Theorem
\ref{T:BFT}.
\end{proof}

\end{document}